\documentclass[12pt,reqno]{amsart}
\usepackage{amssymb,amscd,amsmath,amsthm,color}
\newtheorem{thm}{Theorem}[section]
\newtheorem{prop}[thm]{Proposition}
\newtheorem{lemma}[thm]{Lemma}
\newtheorem{cor}[thm]{Corollary}
\newtheorem{definition}[thm]{Definition}
\newtheorem{remark}[thm]{Remark}
\newtheorem{example}[thm]{Example}

\numberwithin{equation}{section}

\def\cB{\mathcal{B}}
\def\cP{\mathcal{P}}
\def\supp{\mathrm{supp}}
\def\bP{\mathbf{P}}
\def\bd{\mathbf{d}}
\def\ffi{\varphi}
\def\bZ{\mathbb{Z}}
\def\bN{\mathbb{N}}
\def\bR{\mathbb{R}}
\def\1{\mathbf{1}}

\def\eps{\varepsilon}
\def\diam{\mathrm{diam}}

\def\bM{\mathbb{M}}
\def\cA{\mathcal{A}}
\def\bbP{\mathbb{P}}
\def\T{\mathrm{T}}
\def\argmin{\mathop{\mathrm{arg\,min}}}
\def\cH{\mathcal{H}}
\def\cI{\mathcal{I}}
\DeclareMathOperator{\ua}{\uparrow\!}

\def\cX{\mathcal{X}}
\newcommand{\ve}{\varepsilon}
\def\bx{\mathbf{x}}
\def\cC{\mathcal{C}}
\def\cM{\mathcal{M}}
\def\dT{d_\T}
\def\bE{\mathbf{E}}

\begin{document}
\baselineskip=16pt

\allowdisplaybreaks

\title[Convergence theorems]
{Convergence theorems for barycentric maps}
\author[F.~Hiai]{Fumio Hiai}
\address{Tohoku University (Emeritus), Hakusan 3-8-16-303, Abiko 270-1154, Japan}\email{hiai.fumio@gmail.com}
\author[Y.~Lim]{Yongdo Lim}
\address{Department of Mathematics, Sungkyunkwan University, Suwon 440-746, Korea} \email{ylim@skku.edu}
\maketitle

\begin{abstract}
We first develop a theory of conditional expectations for random
variables with values in a complete metric space $M$ equipped with a
contractive barycentric map $\beta$, and then give convergence
theorems for martingales of $\beta$-conditional expectations. We
give the Birkhoff ergodic theorem for $\beta$-values of ergodic
empirical measures and provide a description of the ergodic limit
function in terms of the $\beta$-conditional expectation. Moreover,
we prove the continuity property of the ergodic limit function by
finding a complete metric between contractive barycentric maps on
the Wasserstein space of Borel probability measures on $M$. Finally, the
large derivation property of $\beta$-values of i.i.d.\ empirical
measures is obtained by applying the Sanov large deviation
principle.
\end{abstract}



\section{Introduction and preliminaries}
The main purpose of the present paper is to establish several
convergence theorems for random variables with values in a complete
metric space $(M,d)$ equipped with a contractive barycentric map
$\beta:{\mathcal P}^{p}(M)\to M$, where ${\mathcal P}^{p}(M)$ is the
Wasserstein space of Borel probability measures with finite $p$th
moment. This important class of metric spaces with contractive
barycentric maps contains all Banach spaces, metric spaces that are
nonpositively curved in the weak sense of Busemann, including global
NPC spaces, and convex metric spaces \cite{LL2,LLL}. For instance, a
typical convex metric space is the Banach-Finsler manifold of
positive invertible operators on a Hilbert space equipped with the
Thompson metric. We need no extra condition on the underlying space
$M$, like separability or local compactness,  except only the
existence of a contractive barycentric map $\beta:{\mathcal
P}^{p}(M)\to M$ for some $p\in [1,\infty)$.

As usual, a barycentric map is useful to define expectations of
$p$th integrable $M$-valued random variables via push-forward
measures. However, defining conditional expectations of random
variables with values in a metric space is non-trivial, as
previously discussed by Es-Sahib and Heinich \cite{EH}, Sturm
\cite{St02} and others (as referenced in \cite{EH,St02}). In Section
2, when a probability space is standard Borel, we introduce, by
using the disintegration theorem, the $\beta$-conditional
expectation and derive its fundamental properties including the
contractive and projective properties. We show that our conditional
expectation coincides with Sturm's conditional expectation
\cite{St02} when restricted to the canonical barycentric map on a
global NPC space. In Section 3, motivated by Sturm's martingale
convergence theorem \cite{St02} on a global NPC space, we obtain the
convergence theorem in the sense of $L^p$ and almost everywhere
convergence for $\beta$-martingales of regular type. We also
discuss filtered $\beta$-martingales of Sturm's type.

The most natural problem for contractive barycentric maps is an
extension of the classical Birkhoff ergodic theorem. Ergodic type
results were formerly given in \cite{EH,St} for $L^1$ or $L^2$
i.i.d.\ random variables in nonpositively curved spaces. More
recently, Austin \cite{Au} obtained an $L^2$-ergodic theorem for the
canonical barycentric map on a global NPC space, and Navas \cite{Na}
obtained an $L^1$-ergodic theorem for a specific contractive
barycentric map on a metric space of nonpositive curvature in the
sense of Busemann. The paper \cite{Li} contains an extension of
Navas' ergodic theorem to the parametrized version of the Cartan
barycenter. In Section 4 we review the $L^p$-ergodic theorem in
\cite{Au,Na} for the $\beta$-expectation values of the ergodic
empirical measures in the setting of a general barycentric space
$(M,d,\beta)$. We also provide the description of the ergodic limit
function in terms of the $\beta$-conditional expectation.

There exists many distinct contractive barycentric maps on a fixed
barycentric space $(M,d,\beta)$; for instance, see Remark 6.4 and
Example 6.5 of \cite{St}. In Section 5 we study perturbations for
the ergodic convergence theorem varying over contractive
barycentric maps. We introduce a  complete metric on the set of all
$p$-contractive barycentric maps on $M$ and then  show the
continuity of the ergodic limit function varying over the pairs of
barycentric maps and $p$th integrable random variables.  For the
global NPC space case, we construct a semiflow of contractive
barycentric maps such that the canonical barycentric map plays as a
global attractor fixed point. The convergence of ergodic limits
along any trajectory of barycentric maps to that of the canonical
barycentric map is established as an application of our
$\beta$-convergence theorems.

Finally, in Section 6 we present the large derivation principle for
the $\beta$-values of the empirical measures of $M$-valued i.i.d.\
random variables, which is a stronger version of Sturm's  empirical
law of large numbers \cite{St}.

In order to give precise formulations of the above results, one
needs to recall some backgrounds on measurable $M$-valued functions,
Borel probability measures on $M$, and so on, which are summarized
in the rest of this introductory section.

Let $(M,d)$ be a complete metric space and $\cB(M)$ be the
$\sigma$-algebra of Borel subsets of $M$. Let $\cP(M)$ be the set of
all probability measures on $\cB(M)$ with full support,
  and $\cP_0(M)$ be the set of $\mu\in\cP(M)$ of the form
$\mu=\frac{1}{n}\sum_{j=1}^{n}\delta_{x_j}$ with some $n\in\bN$ and
$x_1,\dots,x_n\in M$. We note \cite{HLL} that every $\mu\in \cP(M)$
has separable support and is the week limit of a sequence of
finitely supported measures. For $1\le p<\infty$ let $\cP^p(M)$ be
the set of $\mu\in\cP(M)$ such that $\int_Md^p(x,y)\,d\mu(y)<\infty$
for some (equivalently, for all) $x\in M$, and $\cP^\infty(M)$ be
the set of $\mu\in\cP(M)$ with bounded support, i.e., $\mu$ is
supported on $\{y\in M:d(x,y)\le\alpha\}$ for some $x\in M$ and some
$\alpha<\infty$. Obviously,
\begin{align}\label{F-1.1}
\cP^1(M)\supset\cP^p(M)\supset\cP^q(M)\supset\cP^\infty(M)\qquad
\mbox{for}\quad1<p<q<\infty.
\end{align}

For $1\le p<\infty$ the {\it $p$-Wasserstein distance} on $\cP^p(M)$ is defined as
$$
d_p^W(\mu,\nu):=\biggl[\inf_{\pi\in\Pi(\mu,\nu)}\int_{M\times M}
d^p(x,y)\,d\pi(x,y)\biggr]^{1/p},\qquad\mu,\nu\in\cP^p(M),
$$
where $\Pi(\mu,\nu)$ denotes the set of $\pi\in\cP(M\times M)$ such that
$\pi(B\times M)=\mu(B)$ and $\pi(M\times B)=\nu(B)$ for all $B\in\cB(M)$.
Moreover, for $p=\infty$ we define
$$
d_\infty^W(\mu,\nu):=\inf_{\pi\in\Pi(\mu,\nu)}\sup\{d(x,y):(x,y)\in\supp(\pi)\},
\qquad\mu,\nu\in\cP^\infty(M).
$$
Note that
\begin{align}\label{F-1.2}
d_1^W\leq d_p^W\leq d_q^W\leq d_\infty^W\qquad
\mbox{for}\quad1<p<q<\infty.
\end{align}
It is well-known \cite{St} that $d_p^W$ is a complete metric on
$\cP^p(M)$ for $1\le p\le\infty$ and $\cP_0(M)$ is dense in
$\cP^p(M)$ for $1\le p<\infty$.

Let $(\Omega,\cA,\bP)$ be a probability space. A Borel measurable
function $\ffi:\Omega\to M$ (i.e., measurable with respect to $\cA$ and
$\cB(M)$ or $M$-valued random variable) is {\it
strongly measurable} if there exists a sequence $\{\ffi_n\}$ of
$M$-valued simple functions, i.e.,
$\ffi_n(\omega)=\sum_{j=1}^{K_n}\1_{A_{n,j}}x_{n,j}$ with
$A_{n,j}\in\cA$ and $x_{n,j}\in M$, such that
$d(\ffi_n(\omega),\ffi(\omega))\to0$ for a.e.\ $\omega\in\Omega$. From the
definition it follows that if $\ffi:\Omega\to M$ is strongly
measurable, then there exists a $\bP$-null set $N\in\cA$ for which
$\{\ffi(\omega):\omega\in\Omega\setminus N\}$ is a separable subset of
$M$ and for any $x\in M$ the function $\omega\in\Omega\setminus
N\mapsto d(x,\ffi(\omega))$ is $\cA$-measurable. Hence the integral
$\int_\Omega d^p(x,\ffi(\omega))\,d\bP(\omega)$
makes sense for any $p\in(0,\infty)$. For each $p\in[1,\infty)$, we
say that a function $\ffi:\Omega\to M$ is {\it $p$th Bochner
integrable} if $\ffi$ is strongly measurable and $\int_\Omega
d^p(x,\ffi(\omega))\,d\bP(\omega)<\infty$ for
some (equivalently, for all) $x\in M$. We denote by
$$
L^p(\Omega;M)=L^p(\Omega,\cA,\bP;M)
$$
the set of all $M$-valued $p$th Bochner integrable functions. We also denote by
$$
L^\infty(\Omega;M)=L^\infty(\Omega,\cA,\bP;M)
$$
the set of all strongly measurable functions $f:\Omega\to M$ such that $d(x,f(\omega))$ is
essentially bounded for some (equivalently, for all) $x\in M$. As usual, for
$\ffi,\psi\in L^p(\Omega;M)$ we consider $\ffi=\psi$ whenever $\ffi(\omega)=\psi(\omega)$
a.e. Obviously,
\begin{align}\label{F-1.3}
L^1(\Omega;M)\supset L^p(\Omega;M)\supset L^q(\Omega;M)\supset L^\infty(\Omega;M)
\qquad\mbox{for}\quad1<p<q<\infty.
\end{align}

The theory of Bochner integrable functions mostly treats measurable functions
with values in a Banach space (see, e.g., \cite{DU}), but basic definitions and results are
valid for measurable functions with values in a complete metric space as well. For instance,
a standard argument gives:

\begin{lemma}\label{L-1.1}
For every $1\le p\le\infty$, the set $L^p(\Omega;M)$ is a complete metric space with the usual
$L^p$-distance
$$
\bd_p(\ffi,\psi):=\biggl[\int_\Omega d^p(\ffi(\omega),\psi(\omega))\,d\bP(\omega)\biggr]^{1/p}
\qquad\mbox{if $1\le p<\infty$},
$$
and
$\bd_\infty(\ffi,\psi):=\mathop{\mathrm{ess\,sup}}_{\omega\in\Omega}
d(\ffi(\omega),\psi(\omega))$ for $p=\infty$. The set of $M$-valued
simple functions is dense in $L^p(\Omega;M)$ for $1\le p<\infty$,
and the set of countably valued functions in $L^\infty(\Omega;M)$ is
dense in $L^\infty(\Omega;M)$.
\end{lemma}

\begin{lemma}\label{L-1.2}
Let $1\le p\le\infty$.
\begin{itemize}
\item[(1)] If $\ffi\in L^p(\Omega;M)$, then the push-forward measure $\ffi_*\bP$ by $\varphi$
belongs to $\cP^p(M)$.
\item[(2)] If $\ffi,\psi\in L^p(\Omega;M)$, then
$d_p^W(\ffi_*\bP,\psi_*\bP)\le\bd_p(\ffi,\psi)$.
\end{itemize}
\end{lemma}

\begin{proof}
(1)\enspace
Let $\ffi\in L^p(\Omega;M)$. There exists a separable closed set $M_0\subset M$ such that
$\ffi(\omega)\in M_0$ for a.e.\ $\omega\in\Omega$. Since
$(\ffi_*\bP)(M_0)=\bP(\ffi^{-1}(M_0))=1$, $\supp(\ffi_*\bP)\subset M_0$ and so
$\ffi_*\bP\in\cP(M)$. Moreover, when $1\le p<\infty$,
$$
\int_Md^p(x,y)\,d(\ffi_*\bP)(y)=\int_\Omega d^p(x,\ffi(\omega))\,d\bP(\omega)<\infty.
$$
When $p=\infty$, we have $d(x,\ffi(\omega))\le\alpha$ a.e.\ for some $\alpha<\infty$,
and hence $\supp(\ffi_*\bP)\subset\{y\in M:d(x,y)\le\alpha\}$.

(2)\enspace
Let $\ffi,\psi\in L^p(\Omega;M)$. Set $\pi:=(\ffi\times\psi)_*\bP$, the push-forward of $\bP$
by the map $\omega\in\Omega\mapsto(\ffi(\omega),\psi(\omega))\in M\times M$. As in the proof of
(1), we have $\pi\in\cP(M\times M)$. For any $B\in\cB(M)$, $\pi(B\times M)=\bP(\ffi^{-1}(B))$
and $\pi(M\times B)=\bP(\psi^{-1}(B))$, so $\pi\in\Pi(\ffi_*\bP,\psi_*\bP)$. Therefore, when
$1\le p<\infty$,
\begin{align*}
d_p^W(\ffi_*\bP,\psi_*\bP)
&\le\biggl[\int_{M\times M}d^p(x,y)\,d(\ffi\times\psi)_*\bP(x,y)\biggr]^{1/p} \\
&=\biggl[\int_\Omega d^p(\ffi(\omega),\psi(\omega))\,d\bP(\omega)\biggr]^{1/p}=\bd_p(\ffi,\psi).
\end{align*}
When $p=\infty$, let $\alpha:=\bd_\infty(\ffi,\psi)$ and
$\Delta:=\{(x,y)\in M\times M:d(x,y)\le\alpha\}$. Then
$\pi(\Delta)=1$, and we have $\supp(\pi)\subset\Delta$, so
$d_\infty^W(\ffi_*\bP,\psi_*\bP)\le\alpha$.
\end{proof}

The following lemma will play an essential role for our purpose. In
fact, a similar inequality follows by specializing \cite[Proposition
7.10]{Vi} to $\mu=\sum_{i=1}^K\alpha_i\delta_{x_i}$ and
$\nu=\sum_{i=1}^K\beta_i\delta_{x_i}$.  The following proof is a
modification (in the specialized situation) of that in \cite{Vi}.

\begin{lemma}\label{L-1.3}
Assume that $1\le p<\infty$. Let $x_1,\dots,x_K\in M$, and $(\alpha_1,\dots,\alpha_K)$ and
$(\beta_1,\dots,\beta_K)$ be probability vectors. Then
$$
d_p^W\Biggl(\sum_{i=1}^K\alpha_i\delta_{x_i},\sum_{i=1}^K\beta_i\delta_{x_i}\Biggr)
\le\Delta\Biggl[{1\over2}\sum_{i=1}^K|\alpha_i-\beta_i|\Biggr]^{1/p},
$$
where $\Delta:=\diam\{x_1,\dots,x_K\}$, the diameter of
$\{x_1,\dots,x_K\}$.
\end{lemma}

\begin{proof}
Let $\gamma_i:=\min\{\alpha_i,\beta_i\}$ for $1\le i\le K$, $I:=\{i:\alpha_i>\gamma_i\}$ and
$J:=\{i:\beta_i>\gamma_i\}$. It is clear that $I\cap J=\emptyset$ and
$$
\sum_{i\in I}(\alpha_i-\gamma_i)=\sum_{i=1}^K(\alpha_i-\gamma_i)
=\sum_{j=1}^K(\beta_j-\gamma_j)=\sum_{j\in J}(\beta_j-\gamma_j)
={1\over2}\sum_{j=1}^K|\alpha_j-\beta_j|.
$$
Let $\rho_{ij}:=(\alpha_i-\gamma_i)(\beta_j-\gamma_j)/\sum_{k\in
J}(\beta_k-\gamma_k)$ for $i\in I$ and $j\in J$; then it is
immediate to check that $ \alpha_i-\gamma_i=\sum_{j\in
J}\rho_{ij}$ ($i\in I$) and $\beta_j-\gamma_j=\sum_{i\in
I}\rho_{ij}$ ($j\in J$).
One can define
$\pi\in\Pi\bigl(\sum_i\alpha_i\delta_{x_i},\sum_i\beta_i\delta_{x_i}\bigr)$
by $ \pi:=\sum_{i=1}^K\gamma_i\delta_{(x_i,x_i)} +\sum_{i\in
I,\,j\in J}\rho_{ij}\delta_{(x_i,x_j)}$. Therefore,
\begin{align*}
d_p^W\Biggl(\sum_{i=1}^K\alpha_i\delta_{x_i},\sum_{i=1}^K\beta_i\delta_{x_i}\Biggr)
&\le\biggl[\int_{M\times
M}d^p(x,y)\,d\pi(x,y)\biggr]^{1/p}=\Biggl[\sum_{i\in I,\,j\in
J}\rho_{ij}d^p(x_i,x_j)\Biggr]^{1/p} \\
&\le\Delta\Biggl[\sum_{i\in I}(\alpha_i-\gamma_i)\Biggr]^{1/p}
=\Delta\Biggl[{1\over2}\sum_{i=1}^K|\alpha_i-\beta_i|\Biggr]^{1/p}.
\end{align*}
\end{proof}

Let $(X,\le)$ be a partially ordered set. For a nonempty subset $A$
of $X$, let $\ua A:=\{y\in X:x\leq y\ \mbox{for some $x\in A$}\}$.
 We say that $A$
is an \emph{upper set} if $\ua A=A.$
Assume that a complete metric space $M$ is equipped with a closed
partial order $\leq$; i.e., $\{(x,y):x\leq y\}$ is closed in
$M\times M$ equipped with the product topology. The \emph{stochastic
order} on $\cP(M)$ introduced in \cite{HLL} is defined by $\mu\leq
\nu$ if $\mu(U)\leq \nu(U)$ for every open upper set $U$,
Several equivalent conditions of $\mu\le\nu$ were given in
\cite{HLL}. We note from \cite{Law,HLL} that for
$\mu=\frac{1}{n}\sum_{j=1}^{n}\delta_{a_{j}}$ and
$\nu=\frac{1}{n}\sum_{j=1}^{n}\delta_{b_{j}}$, $\mu\leq \nu$ if and only
if there exists a permutation $\sigma$ on $\{1,\dots,n\}$ such that
$a_{j}\leq b_{\sigma(j)}$ for all $j=1,\dots,n$.

Assume that $E$ is a real Banach space containing an open convex
cone $C$ such that $\overline C$ is a normal cone (cf.\ \cite{De}).
The cone $\overline C$ defines a closed partial order on $E$ (hence
on $C$) by $x\le y$ if $y-x\in\overline C$. Moreover, $C$ is a
complete metric space with the \emph{Thompson metric} \cite{Th,Nu}
defined by $ \dT(x,y):=\max\{\log M(x/y),\log M(y/x)\},$ where
$M(x/y):=\inf\{\lambda>0:x\le\lambda y\}$. Note that the
$\dT$-topology on $C$ coincides with the relative topology inherited
from $E$. Hence we may consider $\cP(C)$ on $(C,\dT)$. Then it was
shown in \cite{HLL} that the stochastic order on $\cP(C)$ is a
partial order. This is typically the case when $E$ is the algebra
$S(\cH)$ with the operator norm, consisting of self-adjoint bounded linear
operators on a Hilbert space $\cH$, and $C$ is the cone $\bbP(\cH)$
of positive invertible operators on $\cH$.

Now, assume that a complete metric space $M$ is equipped with a
closed partial order. For strongly measurable $M$-values functions
$\ffi,\psi$ on $\Omega$, we define $\ffi\le\psi$ if
$\ffi(\omega)\le\psi(\omega)$ a.e. (The definition makes sense since
$\{\omega:\ffi(\omega)\le\psi(\omega)\}$ is measurable up to a
$\bP$-null set.)

\begin{lemma}\label{L-1.4}
If $\ffi,\psi:\Omega\to M$ are strongly measurable and
$\ffi\le\psi$, then $\ffi_*\bP\le\psi_*\bP$.
\end{lemma}

\begin{proof}
Assume that $\ffi(\omega)\le\psi(\omega)$ for all
$\omega\in\Omega\setminus N$ with a $\bP$-null set $N$. Let $U$ be
an open upper set. If $\omega\in\ffi^{-1}(U)\cap(\Omega\setminus
N)$, then $\ffi(\omega)\in U$ and $\ffi(\omega)\le\psi(\omega)$,
so $\psi(\omega)\in U$. Hence $\ffi^{-1}(U)\cap(\Omega\setminus
N)\subset\psi^{-1}(U)$, so that
$\bP(\ffi^{-1}(U))\le\bP(\psi^{-1}(U))$, implying
$\ffi_*\bP\le\psi_*\bP$.
\end{proof}

\section{Conditional expectations}

In this section, let $1\le p\le\infty$ be fixed, and assume that $\beta:\cP^p(M)\to M$ is a
{\it $p$-contractive barycentric map}, that is, $\beta(\delta_x)=x$ for all $x\in M$ and
\begin{align}\label{F-2.1}
d(\beta(\mu),\beta(\nu))\le d_p^W(\mu,\nu),\qquad\mu,\nu\in\cP^p(M).
\end{align}

\begin{definition}\label{D-2.1}\rm
Let $\ffi\in L^p(\Omega;M)$.
\begin{itemize}
\item[(1)]
Define the {\it $\beta$-expectation} $E^\beta(\ffi)\in M$ of $\ffi$
by
$$
E^\beta(\ffi):=\beta(\ffi_*\bP).
$$
This is well defined by Lemma \ref{L-1.2}\,(1).
\item[(2)] For every $A\in\cA$ with $\bP(A)>0$, consider the reduced probability space
$$
(A,\cA\cap A,\bP_A)\quad\mbox{where}\quad\bP_A:=\bP(A)^{-1}\bP|_{\cA\cap A}.
$$
Let $E^\beta(\ffi|_A)$ be the $\beta$-expectation of $\ffi|_A$ on $(A,\cA\cap A,\bP_A)$, i.e.,
$$
E^\beta(\ffi|_A):=\beta((\ffi|_A)_*\bP_A).
$$
\end{itemize}
\end{definition}

\begin{prop}\label{P-2.2}
Let $\ffi,\psi\in L^p(\Omega;M)$.
\begin{itemize}
\item[(1)] $d(E^\beta(\ffi),E^\beta(\psi))\le\bd_p(\ffi,\psi)$.
\item[(2)] $E^\beta(\1_\Omega x)=x$ for all $x\in M$.
\item[(3)] Assume that $M$ is equipped with a closed partial order and $\beta$ is monotone,
that is, for each $\mu,\nu\in\cP^p(M)$, $\mu\le\nu$ implies
$\beta(\mu)\le\beta(\nu)$. If $\ffi\le\psi$, then $E^\beta(\ffi)\le
E^\beta(\psi)$.
\end{itemize}
\end{prop}

\begin{proof}
(1)\enspace By \eqref{F-2.1} and Lemma \ref{L-1.2}\,(2),
$$
d(E^\beta(\ffi),E^\beta(\psi))=d(\beta(\ffi_*\bP),\beta(\psi_*\bP))
\le d_p^W(\ffi_*\bP,\psi_*\bP)\le\bd_p(\ffi,\psi).
$$

 (2)\enspace Since $(\1_\Omega x)_*\bP=\delta_x$,
$E^\beta(\1_\Omega x)=\beta(\delta_x)=x$.

(3) is obvious from Lemma \ref{L-1.4}.
\end{proof}

\begin{prop}\label{P-2.3}
Let $\ffi,\psi\in L^p(\Omega;M)$. If $E^\beta(\ffi|_A)=E^\beta(\psi|_A)$ for all $A\in\cA$ with
$\bP(A)>0$, then $\ffi=\psi$.
\end{prop}

\begin{proof}
Assume that $\ffi\ne\psi$; then there exists a $\delta>0$ such that
$$
\bP\{\omega\in\Omega:d(\ffi(\omega),\psi(\omega))>\delta\}>0.
$$
One can choose a sequence $\{x_n\}_{n=1}^\infty$ in $M$ such that
$\ffi(\omega),\psi(\omega)\in\overline{\{x_n\}}$ a.e. For
$m,n=1,2,\dots,$ let
$$
A_{m,n}:=\{\omega\in\Omega:d(\ffi(\omega),\psi(\omega))>\delta,\,
d(x_m,\ffi(\omega))\le\delta/4,\,d(x_n,\psi(\omega))\le\delta/4\}.
$$
Since $\bP\bigl(\bigcup_{m,n=1}^\infty A_{m,n}\bigr)
=\bP\{d(\ffi(\omega),\psi(\omega))>\delta\}>0$, one can choose $m,n$ such that
$\bP(A_{m,n})>0$. For $\omega\in A_{m,n}$ we have
$$
\delta<d(\ffi(\omega),\psi(\omega))
\le d(\ffi(\omega),x_m)+d(x_m,x_n)+d(x_n,\psi(\omega))
<d(x_m,x_n)+\delta/2,
$$
so that $d(x_m,x_n)>\delta/2$. Let $\ffi_0:=\1_\Omega x_m$ and $\psi_0:=\1_\Omega x_n$ be
constant functions. For $A=A_{m,n}$ we have $E^\beta(\ffi_0|_A)=x_m$ and
$E^\beta(\psi_0|_A)=x_n$. Moreover, by Proposition 2.2\,(1),
\begin{align*}
d(E^\beta(\ffi|_A),x_m)&=d(E^\beta(\ffi|_A),E^\beta(\ffi_0|_A))
\le\bd_p(\ffi|_A,\ffi_0|_A)\le\delta/4, \\
d(E^\beta(\psi|_A),x_n)&=d(E^\beta(\psi|_A),E^\beta(\psi_0|_A))
\le\bd_p(\psi|_A,\psi_0|_A)\le\delta/4.
\end{align*}
Since $E^\beta(\ffi|_A)=E^\beta(\psi|_A)$ by assumption, we have $d(x_m,x_n)\le\delta/2$, a
contradiction.
\end{proof}

Recall that $(\Omega,\cA)$ is a \emph{standard Borel space} if it is isomorphic to
$(X,\cB(X))$ of a Polish space $X$ and its Borel $\sigma$-algebra $\cB(X)$. In the rest of
this section, unless otherwise stated, we assume that $(\Omega,\cA,\bP)$ is a probability
space over a standard Borel space $(\Omega,\cA)$ and $\cB$ is a sub-$\sigma$-algebra of $\cA$.
To introduce the notion of the $\beta$-conditional expectation with respect to $\cB$, we
utilize the disintegration theorem, which we state as a lemma for convenience. For details see
\cite[Theorem 5.8]{Fu} (where a probability measure space on a standard Borel space is called
a regular measure space).

\begin{remark}\label{R-2.4}\rm
It is known \cite[Corollary 10.4.6]{Bo} that if $X$ is a \emph{Souslin space} (i.e., a
continuous image of a Polish space), then for any probability measure $\bP$ on $\cB(X)$ and
every sub-$\sigma$-algebra $\cB$ of $\cB(X)$ there exists a disintegration of $\bP$
with respect to $\cB$. Thus, the results of this paper when $(\Omega,\cA)$ is a standard
Borel space are also true with a bit weaker assumption that $(\Omega,\cA)$ is isomorphic to
$(X,\cB(X))$ of a Souslin space $X$.
\end{remark}

\begin{lemma}\label{L-2.5}
There exists a family $(\bP_\omega)_{\omega\in\Omega}$ of
probability measures on $(\Omega,\cA)$ such that for every
$A\in\cA$,
\begin{itemize}
\item[(i)]  $\omega\in\Omega\mapsto\bP_\omega(A)$ is
$\cB$-measurable, and
\item[(ii)] $\omega\mapsto\bP_\omega(A)$ is a conditional expectation
of $\1_A$ with respect to $\cB$.
\end{itemize}

Such a family $(\bP_\omega)_{\omega\in\Omega}$ is unique up to a
$\bP$-null set, and moreover it satisfies the following$:$
\begin{itemize}
\item[(iii)] for every $f\in L^1(\Omega;\bR)$, $f\in L^1(\Omega,\cA,\bP_\omega;\bR)$ for
$\bP$-a.e.\ $\omega\in\Omega$ and $\omega\mapsto\int_\Omega f(\tau)\,d\bP_\omega(\tau)$ is a
conditional expectation of $f$ with respect to $\cB$. In particular,
$$
\int_\Omega f\,d\bP=\int_\Omega\biggl[\int_\Omega f(\tau)\,d\bP_\omega(\tau)\biggr]d\bP(\omega).
$$
\end{itemize}
\end{lemma}

The family $(\bP_\omega)_{\omega\in\Omega}$ given in the above lemma
is called a {\it disintegration} of $\bP$ with respect to $\cB$. The
next lemma is easily seen from the primary property (ii) of the
above lemma, while we supply the proof for completeness.

\begin{lemma}\label{L-2.6}
Let $(\bP_\omega)_{\omega\in\Omega}$ be a disintegration of $\bP$
with respect to $\cB$. For every $\psi\in
L^1(\Omega,\cB,\bP;M)$, there is a $\bP$-null set $N\in\cB$ such
that for every $\omega\in\Omega\setminus N$, $\psi(\tau)$ is
constant for $\bP_\omega$-a.e.\ $\tau\in\Omega$.
\end{lemma}

\begin{proof}
Note that $\psi(\tau)$ is constant for $\bP_\omega$-a.e.\ if and
only if $\psi_*\bP_\omega$ is singly supported. Choose a countable
set $\{x_i\}_{i=1}^\infty$ in $M$ such that
$\psi(\omega)\in\overline{\{x_i\}}$ a.e. For $x\in M$ and $k\in\bN$,
set $U_{1/k}(x):=\{y\in M:d(y,x)<1/k\}$, the open ball of center $x$
and radius $1/k$. Let $\{(U_n,V_n)\}_{n=1}^\infty$ be an enumeration
of all pairs $(U_{1/k}(x_i),U_{1/k}(x_j))$ such that
$U_{1/k}(x_i)\cap U_{1/k}(x_j)=\emptyset$ with $i,j,k\in\bN$. Then
it is easy to see that $\psi_*\bP_\omega$ is singly supported if and
only if $(\psi_*\bP_\omega)(U_n)\cdot(\psi_*\bP_\omega)(V_n)=0$ for
all $n$, that is,
$\bP_\omega(\psi^{-1}(U_n))\cdot\bP_\omega(\psi^{-1}(V_n))=0$ for
all $n$. Since $\psi^{-1}(U_n),\psi^{-1}(V_n)\in\cB$, it follows
from Lemma \ref{L-2.5}\,(ii) that
$$
\bP_\omega(\psi^{-1}(U_n))=\1_{\psi^{-1}(U_n)}(\omega),\quad
\bP_\omega(\psi^{-1}(V_n))=\1_{\psi^{-1}(V_n)}(\omega)\ \ \mbox{a.e.}
$$
so that $\bP_\omega(\psi^{-1}(U_n))\cdot\bP_\omega(\psi^{-1}(V_n))=0$ a.e. Hence there is a
$\bP$-null set $N\in\cB$ such that for every $\omega\in\Omega\setminus N$ we have
$\bP_\omega(\psi^{-1}(U_n))\cdot\bP_\omega(\psi^{-1}(V_n))=0$ for all $n$, so $\psi(\tau)$ is
constant for $\bP_\omega$-a.e.
\end{proof}

\begin{lemma}\label{L-2.7}
Let $(\bP_\omega)_{\omega\in\Omega}$ be a disintegration of $\bP$ with respect to $\cB$. Let
$1\le p<\infty$.
\begin{itemize}
\item[(1)] If $\ffi\in L^p(\Omega;M)$, then there is a $\bP$-null $N\in\cB$ such that
$\ffi\in L^p(\Omega,\cA,\bP_\omega;M)$ and $\ffi_*\bP_\omega\in\cP^p(M)$ for all
$\omega\in\Omega\setminus N$.
\item[(2)] If $\ffi,\psi\in L^p(\Omega;M)$, then there is a $\bP$-null $N\in\cB$ such that
\begin{align}\label{F-2.2}
d(\beta(\ffi_*\bP_\omega),\beta(\psi_*\bP_\omega))
\le\biggl[\int_\Omega d^p(\ffi(\tau),\psi(\tau))\,d\bP_\omega(\tau)\biggr]^{1/p},
\qquad\omega\in\Omega\setminus N.
\end{align}
\end{itemize}
\end{lemma}

\begin{proof}
(1)\enspace Let $\ffi\in L^p(\Omega;M)$ and $x\in M$. Since
$\omega\mapsto d^p(x,\ffi(\omega))$ is in
$L^1(\Omega;{\Bbb R})$, it follows from Lemma \ref{L-2.5}\,(iii) that
$$
\int_\Omega\biggl[\int_\Omega d^p(x,\ffi(\tau))\,d\bP_\omega(\tau)\biggr]d\bP(\omega)
=\int_\Omega d^p(x,\ffi(\omega))\,d\bP(\omega)<\infty.
$$
Hence there is a $\bP$-null $N\in\cB$ such that for every $\omega\in\Omega\setminus N$ we have
$$
\int_\Omega d^p(x,\ffi(\tau))\,d\bP_\omega(\tau)<\infty,\quad\mbox{i.e.},\quad
\ffi\in L^p(\Omega,\cA,\bP_\omega;M)
$$
so that $\ffi_*\bP_\omega\in\cP^p(M)$ by Lemma \ref{L-1.2}\,(1).

(2)\enspace
Let $\ffi,\psi\in L^p(\Omega;M)$. By (1) there is a $\bP$-null $N\in\cB$ such that
$\ffi_*\bP_\omega,\psi_*\bP_\omega\in\cP^p(M)$ for all $\omega\in\Omega\setminus N$. For such
$\omega$, by (2.1) and Lemma \ref{L-1.2}\,(2) (applied to $\bP_\omega$
in place of $\bP$) we have
$$
d(\beta(\ffi_*\bP_\omega),\beta(\psi_*\bP_\omega))
\le d_p^W(\ffi_*\bP_\omega,\psi_*\bP_\omega)
\le\biggl[\int_\Omega d^p(\ffi(\tau),\psi(\tau))\,d\bP_\omega(\tau)\biggr]^{1/p}.
$$
\end{proof}

Now, assume that $1\le p<\infty$ and $\beta:\cP^p(M)\to M$ is a $p$-contractive barycentric map.

\begin{definition}\label{D-2.8}\rm
By using the disintegration $(\bP_\omega)_{\omega\in\Omega}$ of $\bP$ with respect to $\cB$,
for each $\ffi\in L^p(\Omega,\cA,\bP;M)$, define the {\it $\beta$-conditional expectation} of
$\ffi$ with respect to $\cB$ by
$$
E_\cB^\beta(\ffi)(\omega):=\beta(\ffi_*\bP_\omega),\qquad\omega\in\Omega.
$$
\end{definition}

The above definition makes sense by Lemma \ref{L-2.7}\,(1) but the
$\cB$-strong measurability of $E_\cB^\beta(\ffi)$ is proved in (1)
of the next theorem. This implies that the left-hand side of
\eqref{F-2.2} is a $\cB$-measurable function of $\omega$, while the
measurability of the right-hand side is seen from Lemma
\ref{L-2.5}\,(iii). The following shows in
particular   that the conditional expectation
$E_\cB^\beta:L^p(\Omega,\cA,\bP;M)\to L^p(\Omega,\cB,\bP;M)$ is
well defined and is a contractive retraction.

\begin{thm}\label{T-2.9}
Let $\ffi,\psi\in L^p(\Omega;M)$.
\begin{itemize}
\item[(1)] $E_\cB^\beta(\ffi)\in L^p(\Omega,\cB,\bP;M)$.
\item[(2)] $\bd_p(E_\cB^\beta(\ffi),E_\cB^\beta(\psi))\le\bd_p(\ffi,\psi)$.
\item[(3)] $\ffi\in L^p(\Omega,\cB,\bP;M)$ if and only if $E_\cB^\beta(\ffi)=\ffi$. Hence
$E_\cB^\beta(E_\cB^\beta(\ffi))=E_\cB^\beta(\ffi)$.
\item[(4)] When $\cB=\{\emptyset,\Omega\}$, $E_\cB^\beta(\ffi)=E^\beta(\ffi)$.
\item[(5)] Assume that $M$ is equipped with a closed partial order and $\beta$ is monotone.
If $\ffi\le\psi$, then $E_\cB^\beta(\ffi)\le E_\cB^\beta(\psi)$.
\end{itemize}
\end{thm}

\begin{proof}
(1)\enspace First, assume that $\ffi$ is a simple function, i.e.,
$\ffi=\sum_{j=1}^K\1_{A_j}x_j$, where $\{A_1,\dots,A_K\}$ is a measurable
partition of $\Omega$. Since
$\ffi_*\bP_\omega=\sum_{j=1}^K\bP_\omega(A_j)\delta_{x_j}$, one has
\begin{align}\label{F-2.3}
E_\cB^\beta(\ffi)(\omega)=\beta\Biggl(\sum_{j=1}^K\bP_\omega(A_j)\delta_{x_j}\Biggr).
\end{align}
By Lemma \ref{L-2.5}\,(ii) one has $\sum_{j=1}^K\bP_\omega(A_j)=1$
for all $\omega\in\Omega\setminus N$ with a $\bP$-null set
$N\in\cB$. For each $k\in\bN$ and $\omega\in\Omega\setminus N$,
approximating $\bP_\omega(A_j)$ ($1\le j\le K$) with numbers $l/k$
($0\le l\le k$) one can choose sequences $\{\xi_{jk}\}_{k=1}^\infty$
($1\le j\le K$) of $\cB$-simple functions $\xi_{jk}:\Omega\to[0,1]$
such that $\sum_{j=1}^K\xi_{jk}(\omega)=1$ for all $\omega\in\Omega$
and $k\in\bN$, and
$\mathrm{ess\,sup}_{\omega\in\Omega}\,|\xi_{jk}(\omega)-\bP_\omega(A_j)|\le1/k$
for $1\le j\le K$. Then one has by \eqref{F-2.3}, \eqref{F-2.1} and
Lemma \ref{L-1.3},
\begin{align}
d\Biggl(\beta\Biggl(\sum_{j=1}^K\xi_{jk}(\omega)\delta_{x_j}\Biggr),
E_\cB^\beta(\ffi)(\omega)\Biggr) \nonumber &\le
d_p^W\Biggl(\sum_{j=1}^K\xi_{jk}(\omega)\delta_{x_j},
\sum_{j=1}^K\bP_\omega(A_j)\delta_{x_j}\Biggr) \nonumber\\
&\le\Delta\Biggl[\sum_{j=1}^K\big|\xi_{jk}(\omega)-\bP_\omega(A_j)\big|
\Biggr]^{1/p}\,\longrightarrow\,0\quad\mbox{a.e.} \label{F-2.4}
\end{align}
as $k\to\infty$, where $\Delta:=\diam\{x_1,\dots,x_K\}$. It is clear that
$\beta\bigl(\sum_{j=1}^K\xi_{jk}(\omega)\delta_{x_j}\bigr)$'s are
$\cB$-simple functions. Hence $E_\cB^\beta(\ffi)$ is $\cB$-strongly
measurable. Moreover, for each $k,l\in\bN$, by \eqref{F-2.1} and
Lemma \ref{L-1.3} again,
\begin{align*}
d\Biggl(\beta\Biggl(\sum_{j=1}^K\xi_{jk}(\omega)\delta_{x_j}\Biggr),
\beta\Biggl(\sum_{j=1}^K\xi_{jl}(\omega)\delta_{x_j}\Biggr)\Biggr)&\le
d_p^W\Biggl(\sum_{j=1}^K\xi_{jk}(\omega)\delta_{x_j},
\sum_{j=1}^K\xi_{jl}(\omega)\delta_{x_j}\Biggr) \\
&\le\Delta\Biggl[\sum_{j=1}^K|\xi_{jk}(\omega)-\xi_{jl}(\omega)|
\Biggr]^{1/p},
\end{align*}
so that for whichever $p\in[1,\infty)$,
\begin{align*}
\bd_p\Biggl(\beta\Biggl(\sum_{j=1}^K\xi_{jk}(\cdot)\delta_{x_j}\Biggr),
\beta\Biggl(\sum_{j=1}^K\xi_{jl}(\cdot)\delta_{x_j}\Biggr)\Biggr)
\le\Delta\Biggl[\sum_{j=1}^K
\mathop{\mathrm{ess\,sup}}_{\omega\in\Omega}|\xi_{jk}(\omega)-\xi_{jl}(\omega)|\Biggr]^{1/p}
\longrightarrow\,0
\end{align*}
as $k,l\to\infty$. Therefore,
$\beta\bigl(\sum_{j=1}^K\xi_{jk}(\cdot)\delta_{x_j}\bigr)$ converges
in $\bd_p$ as $k\to\infty$ to an element of $L^p(\Omega,\cB,\bP;M)$.
Since the limit must be $E_\cB^\beta(\ffi)$ due to \eqref{F-2.4}, it
follows that $E_\cB^\beta(\ffi)\in L^p(\Omega,\cB,\bP;M)$ when
$\ffi$ is a simple function.

Next, for general $\ffi\in L^p(\Omega;M)$ choose a
sequence $\{\ffi_k\}_{k=1}^\infty$ of simple functions in
$L^p(\Omega;M)$ such that $\bd_p(\ffi_k,\ffi)\to0$,
due to the denseness of $M$-valued simple functions.
Then $E_\cB^\beta(\ffi_k)\in L^p(\Omega,\cB,\bP;M)$ as proved above.
By Lemmas \ref{L-2.7}\,(2) and \ref{L-2.5}\,(iii),
\begin{align}
\bd_p^p(E_\cB^\beta(\ffi_k),E_\cB^\beta(\ffi_l))
&=\int_\Omega d^p(\beta((\ffi_k)_*\bP_\omega),\beta((\ffi_l)_*\bP_\omega))\,d\bP(\omega)
\nonumber\\
&\le\int_\Omega\biggl[\int_\Omega d^p(\ffi_k(\tau),\ffi_l(\tau))\,d\bP_\omega(\tau)\biggr]
d\bP(\omega) \nonumber\\
&=\bd_p^p(\ffi_k,\ffi_l)\ \longrightarrow\ 0\quad\mbox{as
$k,l\to\infty$}. \label{F-2.5}
\end{align}
Moreover, by Lemma \ref{L-2.7}\,(2) there is a $\bP$-null $N_0\in\cB$ such that
\begin{align*}
d(E_\cB^\beta(\ffi_k)(\omega),E_\cB^\beta(\ffi)(\omega))
&=d(\beta((\ffi_k)_*\bP_\omega),\beta(\ffi_*\bP_\omega)) \\
&\le\biggl[\int_\Omega d^p(\ffi_k(\tau),\ffi(\tau))\,d\bP_\omega(\tau)\biggr]^{1/p},
\qquad\omega\in\Omega\setminus N_0.
\end{align*}
Now, let $\zeta_k(\omega):=\int_\Omega
d^p(\ffi_k(\tau),\ffi(\tau))\,d\bP_\omega(\tau)$ for
$\omega\in\Omega$. Then using Lemma \ref{L-2.5}\,(iii)
to the function $d^p(\ffi_k(\omega),\ffi(\omega))$ we have
$\zeta_k\in L^1(\Omega,\cB,\bP;\bR)$ and
$$
\int_\Omega\zeta_k(\omega)\,d\bP(\omega)
=\bd_p^p(\ffi_k,\ffi)\ \longrightarrow\ 0 \quad\mbox{as
$k\to\infty$}.
$$
Hence, by choosing a subsequence of $\{\zeta_k\}$ we may assume that $\zeta_k(\omega)\to0$
a.e.\ (see \cite[p.\,93, Theorem D]{Hal}), so there is a $\bP$-null $N_1\in\cB$ such that
$\lim_{k\to\infty}\zeta_k(\omega)=0$ for all $\omega\in\Omega\setminus N_1$.
Furthermore, for every $k\in\bN$, since $E_\cB^\beta(\ffi_k)$ is $\cB$-strongly measurable,
one can choose a $\cB$-simple function $\psi_k$ and a $B_k\in\cB$ such that $\bP(B_k)<1/k^2$
and $d(E_\cB^\beta(\ffi_k),\psi_k(\omega))<1/k$ for all $\omega\in\Omega\setminus B_k$. Set
$N:=N_0\cup N_1\cup(\limsup_kB_k)\in\cB$. Then $\bP(N)=0$ by the Borel-Cantelli lemma, and for
every $\omega\in\Omega\setminus N$, we have $\omega\in\Omega\setminus N_0$,
$\omega\in\Omega\setminus N_1$ and $\omega\in\Omega\setminus B_k$ for all $k$ sufficiently
large, so that
\begin{align*}
d(E_\cB^\beta(\ffi)(\omega),\psi_k(\omega))
&\le d(E_\cB^\beta(\ffi)(\omega),E_\cB^\beta(\ffi_k)(\omega))
+d(E_\cB^\beta(\ffi_k)(\omega),\psi_k(\omega)) \\
&\le\zeta_k(\omega)^{1/p}+{1\over k}\ \longrightarrow\ 0\quad\mbox{as $k\to\infty$}.
\end{align*}
This implies that $E_\cB^\beta(\ffi)$ is $\cB$-strongly measurable
and $E_\cB^\beta(\ffi_k)(\omega)\to E_\cB^\beta(\ffi)(\omega)$ a.e.
From this and \eqref{F-2.5} we find that $E_\cB^\beta(\ffi)$ is the
$\bd_p$-limit of $E_\cB^\beta(\ffi_k)$ and hence (1) follows.

(2)\enspace
The proof is similar to that of the inequality in \eqref{F-2.5}.

(3)\enspace If $E_\cB^\beta(\ffi)=\ffi$, then $\ffi\in
L^p(\Omega,\cB,\bP;M)$ by (1). Conversely, assume that $\ffi\in
L^p(\Omega,\cB,\bP;M)$. By approximation, we may assume that $\ffi$
is a $\cB$-simple function, i.e., $\ffi=\sum_{j=1}^K\1_{B_j}x_j$
with a $\cB$-partition $\{B_1,\dots,B_n\}$ of $\Omega$, so
$E_\cB^\beta(\ffi)=\beta\bigl(\sum_{j=1}^K\bP_\omega(B_j)\delta_{x_j}\bigr)$.
Since $\bP_\omega(B_j)=\1_{B_j}(\omega)$ a.e.\ by Lemma \ref{L-2.5}\,(ii), we have
$$
E_\cB^\beta(\ffi)(\omega)=\sum_{j=1}^K\1_{B_j}(\omega)\beta(\delta_{x_j})
=\sum_{j=1}^K\1_{B_j}(\omega)x_j=\ffi(\omega)\ \ \mbox{a.e.}
$$

(4) is obvious and (5) follows from Lemma \ref{L-1.4}.
\end{proof}

\begin{remark}\label{R-2.10}\rm
The last paragraph of the above proof of (1) may be a bit
complicated. A simpler way to construct the map $\ffi\in
L^p(\Omega;M)\mapsto E_\cB^\beta(\ffi)\in L^p(\Omega,\cB,\bP;M)$ is
as follows: For a simple function $\ffi:\Omega\to M$,
$E_\cB^\beta(\ffi)\in L^p(\Omega,\cB,\bP;M)$ is well defined as
above. For every simple functions $\ffi,\psi$ we have
$\bd_p(E_\cB^\beta(\ffi),E_\cB^\beta(\psi))\le\bd_p(\ffi,\psi)$ as
in \eqref{F-2.5}. Hence the map $E_\cB^\beta$ on the simple
functions can uniquely extend to $E_\cB^\beta$ on $L^p(\Omega;M)$ by
continuity. However, this abstract definition does not imply the
$\cB$-strong measurability of
$\omega\mapsto\beta(\ffi_*\bP_\omega)$, so the expression
$E_\cB^\beta(\ffi)(\omega)=\beta(\ffi_*\bP_\omega)$ (Definition
\ref{D-2.8}) is not clear.
\end{remark}

\begin{remark}\label{R-2.11}\rm
Assume that $1\le p_0<\infty$ and $\beta:\cP^{p_0}(M)\to M$ is a
$p_0$-contractive barycentric map. Then in view of \eqref{F-1.1} and
\eqref{F-1.2} we note that for every $p\in[p_0,\infty]$,
$\beta|_{\cP^p(M)}:\cP^p(M)\to M$ is a $p$-contractive barycentric
map. It follows from this and \eqref{F-1.3} that Theorem \ref{T-2.9}
holds for every $p\in[p_0,\infty)$. Moreover, when $\ffi\in
L^\infty(\Omega;M)$ and $\ffi_0:=\1_\Omega x_0\in M$, one has
$$
\bd_p(E_\cB^\beta(\ffi),\ffi_0)=\bd_p(E_\cB^\beta(\ffi),E_\cB^\beta(\ffi_0))
\le\bd_p(\ffi,\ffi_0),\qquad p_0\le p<\infty,
$$
whose limit as $p\to\infty$ gives
$\bd_\infty(E_\cB^\beta(\ffi),\ffi_0)\le\bd_\infty(\ffi,\ffi_0)<\infty$ so that
$E_\cB^\beta(\ffi)\in L^\infty(\Omega,\cB,\bP;M)$. Also, for $\ffi,\psi\in L^\infty(\Omega;M)$,
$$
\bd_\infty(E_\cB^\beta(\ffi),E_\cB^\beta(\psi))
=\lim_{p\to\infty}\bd_p(E_\cB^\beta(\ffi),E_\cB^\beta(\psi))
\le\lim_{p\to\infty}\bd_p(\ffi,\psi)=\bd_\infty(\ffi,\psi).
$$
Therefore, Theorem \ref{T-2.9} holds for $p=\infty$ as well in this situation. However, it is
not clear whether Theorem \ref{T-2.9} holds for $p=\infty$ when an $\infty$-contractive
barycentric map $\beta:\cP^\infty(M)\to M$ is given. Note that the proof of the theorem
heavily relies on Lemma \ref{L-1.3} and the assumption $1\le p<\infty$ is essential for Lemma
\ref{L-1.3}. So, when $p=\infty$, it does not seem easy to prove that the function
$\omega\mapsto\beta(\ffi_*\bP_\omega)$ is $\cB$-strongly measurable.
\end{remark}

\begin{example}\label{E-2.12}\rm
An important property of the conventional conditional expectation is
the associativity $E_{\mathcal C}\circ E_\cB=E_{\mathcal
C}$ for sub-$\sigma$-algebras $\mathcal{C}\subset\cB\subset\cA$.
However, this fails to hold for the $\beta$-conditional expectation.
To give a counter-example, let $M=\mathbb{P}_n$ be the Cartan-Hadamard
manifold of $n\times n$ positive definite matrices equipped with
the trace metric $ds = \| A^{-1/2} dA\, A^{-1/2} \|_{2} =
\left[{\mathrm{tr}}(A^{-1} dA)^{2}\right]^{1/2}$, and $\beta=G$ be
the Cartan barycenter (or the Karcher mean):
$$
G(\mu)=\argmin_{Z\in {\Bbb P}_n}\int_{{\Bbb
P}_n}\bigl[d^2(Z,X)-d^2(Y,X)\bigr]\,d\mu(X).
$$
Let $\Omega=\{1,2,3\}$, $\cA=2^\Omega$, and
$\bP=(p_1,p_2,p_3)$. Let $\cB=\{\emptyset,\{1\},\{2,3\},\Omega\}$.
Let $\ffi=\sum_{j=1}^3\1_{\{j\}}A_j$ with $A_j\in\mathbb{P}_n$. Then
we have for $S\in\cA$,
$$
\bP_1(S)={\bP(S\cap\{1\})\over p_1},\qquad
\bP_2(S)=\bP_3(S)={\bP(S\cap\{2,3\})\over p_2+p_3}.
$$
Therefore,
$$
\ffi_*\bP_1=\delta_{A_1},\qquad
\ffi_*\bP_2=\ffi_*\bP_3={p_2\over p_2+p_3}\delta_{A_2}+{p_3\over p_2+p_3}\delta_{A_3},
$$
so that $E_\cB^G(\ffi)(1)=G(\delta_{A_1})=A_1$ and
\begin{align*}
&E_\cB^G(\ffi)(2)=E_\cB^G(\ffi)(3) =G\biggl({p_2\over
p_2+p_3}\delta_{A_2}+{p_3\over p_2+p_3}\delta_{A_3}\biggr)
=A_2\#_{p_3/(p_2+p_3)}A_3,
\end{align*}
where $t\mapsto
A\#_tB:=A^{1/2}(A^{-1/2}BA^{-1/2})^tA^{1/2}$ is  the unique (up to
parametrization) geodesic joining $A$ and $B$ (cf.\ \cite{Bh}).  Now we show that
$E^G\ne E^G\circ E_\cB^G$ (note that $E^G=E_{\mathcal C}^G$ with
${\mathcal C}=\{\emptyset,\Omega\}$ by Theorem \ref{T-2.9}\,(4)).
Assume on the contrary that $E^G=E^G\circ E_\cB^G$; then
$$
G(p_1\delta_{A_1}+p_2\delta_{A_2}+p_3\delta_{A_3})
=A_1\#_{p_2+p_3}(A_2\#_{p_3/(p_2+p_3)}A_3),
$$
holds for all $A_1,A_2,A_3\in\mathbb{P}_n$ and all probabilities
$(p_1,p_2,p_3)$. Then we must have
\begin{align*}
A_1\#_{p_2+p_3}(A_2\#_{p_3/(p_2+p_3)}A_3)
&=G(p_1\delta_{A_1}+p_2\delta_{A_2}+p_3\delta_{A_3}) \\
&=G(p_2\delta_{A_2}+p_1\delta_{A_{1}}+p_3\delta_{A_{3}})
=A_2\#_{p_1+p_3}(A_1\#_{p_3/(p_1+p_3)}A_3).
\end{align*}
In particular, when $p_1=p_2=p_3=1/3$ and $A_3=I$, the above becomes
$A_1\#_{2/3}A_2^{1/2}=A_2\#_{2/3}A_1^{1/2}$ or $A_1\#_{2/3}A_2^{1/2}=A_1^{1/2}\#_{1/3}A_2$.
Since this certainly fails to hold, we have a contradiction.

In view of Theorem \ref{T-2.13} below, Sturm's example in
\cite[Example 3.2]{St02} on the $3$-spider serves as another
counter-example to the associativity of the $\beta$-conditional
expectation.
\end{example}

From Example \ref{E-2.12} we find that the following characterization of $E_\cB^\beta(\ffi)$ of
$\ffi\in L^p(\Omega;M)$ like the conventional conditional expectation is not valid:
$$
\psi=E_\cB^\beta(\ffi)\ \iff\ \begin{cases}
\psi\in L^p(\Omega,\cB,\bP;M)\ \mbox{and} \\
E^\beta(\psi|_B)=E^\beta(\ffi|_B)\ \mbox{for all $B\in\cB$ with $\bP(B)>0$}.
\end{cases}
$$

 Finally, we specialize our conditional expectation
to the case of a \emph{global NPC space} (alternatively, \emph{CAT$(0)$} or
\emph{Hadamard space}). Let $(M,d)$ is a global NPC space. The
\emph{canonical barycentric map} $\lambda$ on $\cP^1(M)$ defined in
\cite{St} is
\begin{eqnarray}\label{E:least}
\lambda(\mu):=\argmin_{z\in
M}\int_M\bigl[d^2(z,x)-d^2(y,x)\bigr]\,d\mu(x)
\end{eqnarray}
for each $\mu\in\cP^1(M)$ independently of the choice of $y\in M$.
If $\mu\in\cP^2(M)$, then $\lambda(\mu)$ is more simply given by
$$
\lambda(\mu)=\argmin_{z\in M}\int_Md^2(z,x)\,d\mu(x).
$$

Assume that $(\Omega,\cA,\bP)$ be a \emph{general} probability space and $\cB$ is a
sub-$\sigma$-algebra of $\cA$. In \cite{St02} Sturm introduced, for each
$\ffi\in L^2(\Omega;M)$, the conditional expectation $\bE_\cB(\ffi)\in L^2(\Omega,\cB,\bP;M)$
of $\ffi$ with respect to $\cB$ as
$$
\bE_\cB(\ffi):=\argmin_{\psi\in L^2(\Omega,\cB,\bP;M)}\bd_2(\ffi,\psi).
$$
He then proved that for every $\ffi,\psi\in L^2(\Omega;M)$,
$$
d(\bE_\cB(\ffi)(\omega),\bE_\cB(\psi)(\omega))\le E_\cB[d(\ffi,\psi)](\omega)\ \ \mbox{a.e.},
$$
where $E_\cB[d(\ffi,\psi)]$ is the usual conditional expectation of the function
$d(\ffi(\omega),\psi(\omega))$ with respect to $\cB$. From this he showed that
$\bE_\cB$ extends continuously from $L^2(\Omega;M)$ to $L^1(\Omega;M)$ and that for every
$p\in[1,\infty]$, $\bE_\cB$ maps $L^p(\Omega;M)$ into $L^p(\Omega,\cB,\bP;M)$ and
\begin{align}\label{F-2.6}
\bd_p(\bE_\cB(\ffi),\bE_\cB(\psi))\le\bd_p(\ffi,\psi),\qquad\ffi,\psi\in L^p(\Omega;M).
\end{align}

Now, we assume that $(\Omega,\cA)$ is a standard Borel space. Our definition then provides
the conditional expectation $E_\cB^\lambda(\ffi)\in L^1(\Omega,\cB,\bP;M)$ for every
$\ffi\in L^1(\Omega;M)$, and by Remark \ref{R-2.11} for every $p\in[1,\infty]$, $E_\cB^\lambda$
maps $L^p(\Omega;M)$ into $L^p(\Omega,\cB,\bP;M)$ and
\begin{align}\label{F-2.7}
\bd_p(E_\cB^\lambda(\ffi),E_\cB^\lambda(\psi))\le\bd_p(\ffi,\psi),\qquad
\ffi,\psi\in L^p(\Omega;M).
\end{align}

Sturm's conditional expectation is restricted to a global NPC space $(M,d)$ while
$(\Omega,\cA,\bP)$ is general. On the other hand, our definition needs a restriction on
$(\Omega,\cA)$ to guarantee the existence of a disintegration, while it can be applied to
a general contractive barycentric map. The next theorem says that Sturm's conditional
expectation and ours are the same, in the situation where both can be defined.

\begin{thm}\label{T-2.13}
Assume that $(\Omega,\cA,\bP)$ is a standard Borel probability space. Let
$(M,d)$ be a global NPC space, and $\lambda$ be given as above. Then
for every $p\in[1,\infty]$ and every $\ffi\in L^p(\Omega;M)$,
$$
\bE_\cB(\ffi)=E_\cB^\lambda(\ffi).
$$
\end{thm}

\begin{proof}
First, assume that $p=2$ and $\ffi\in L^2(\Omega;M)$. By Lemma \ref{L-2.6} there is a
$\bP$-null set $N\in\cB$ such that for every $\omega\in\Omega\setminus N$, both
$E_\cB^\lambda(\tau)$ and $\bE_\cB(\ffi)(\omega)$ are constant
$\bP_\omega$-a.e.\ $\tau\in\Omega$. Hence, for every $\omega\in\Omega\setminus N$,
letting
$$
E_\cB^\lambda(\ffi)(\tau)=\lambda(\ffi_*\bP_\omega)=x,\quad
\bE_\cB(\ffi)(\tau)=z\quad\mbox{$\bP_\omega$-a.e.},
$$
we have
\begin{align*}
\int_\Omega d^2(E_\cB^\lambda(\ffi)(\tau),\ffi(\tau))\,d\bP_\omega(\tau)
&=\int_\Omega d^2(x,\ffi(\tau))\,d\bP_\omega(\tau)
=\int_Md^2(\lambda(\ffi_*\bP_\omega),y)\,d\ffi_*\bP_\omega(y) \\
&\le\int_Md^2(z,y)\,d\ffi_*\bP_\omega(y)
=\int_\Omega d^2(z,\ffi(\tau))\,d\bP_\omega(\tau) \\
&=\int_\Omega d^2(\bE_\cB(\ffi)(\tau),\ffi(\tau))\,d\bP_\omega(\tau).
\end{align*}
Therefore, we have by Lemma \ref{L-2.5}\,(iii)
\begin{align*}
\bd_2^2(E_\cB^\lambda(\ffi),\ffi)
&=\int_\Omega\biggl[\int_\Omega d^2(E_\cB^\lambda(\ffi)(\tau),\ffi(\tau))
\,d\bP_\omega(\tau)\biggr]d\bP(\omega) \\
&\le\int_\Omega\biggl[\int_\Omega d^2(\bE_\cB(\ffi)(\tau),\ffi(\tau))
\,d\bP_\omega(\tau)\biggr]d\bP(\omega)=\bd_2^2(\bE_\cB(\ffi),\ffi).
\end{align*}
Hence $\bE_\cB(\ffi)=E_\cB^\lambda(\ffi)$ follows by definition of $\bE_\cB(\ffi)$.

Next, let $p\in[1,\infty]$ be arbitrary and $\ffi\in L^p(\Omega;M)$. One can choose a
sequence $\{\ffi_k\}$ in $L^\infty(\Omega;M)$ ($\subset L^2(\Omega;M)$) such that
$\bd_p(\ffi_k,\ffi)\to0$. Since $\bE_\cB(\ffi_k)=E_\cB^\lambda(\ffi_k)$ for all $k$ by the
above case, one has by \eqref{F-2.6} and \eqref{F-2.7}
\begin{align*}
\bd_p(\bE_\cB(\ffi),E_\cB^\lambda(\ffi))
&\le\bd_p(\bE_\cB(\ffi),\bE_\cB(\ffi_k))+\bd_p(E_\cB^\lambda(\ffi_k),E_\cB^\lambda(\ffi)) \\
&\le2\bd_p(\ffi_k,\ffi)\,\longrightarrow\,0,
\end{align*}
and hence $\bE_\cB(\ffi)=E_\cB^\lambda(\ffi)$.
\end{proof}

\section{Martingale convergence theorem}

Let $(\Omega,\cA,\bP)$ be a probability space on a standard Borel
space $(\Omega,\cA)$. Let $\{\cB_n\}_{n=1}^\infty$ be a sequence of
sub-$\sigma$-algebras of $\cA$ such that either
$\cB_1\subset\cB_2\subset\cdots$ or
$\cB_1\supset\cB_2\supset\cdots$. Then let $\cB_\infty$ be the
sub-$\sigma$-algebra of $\cA$ generated by
$\bigcup_{n=1}^\infty\cB_n$ in the increasing case and
$\cB_\infty:=\bigcap_{n=1}^\infty\cB_n$ in the decreasing case. Let
$1\le p<\infty$ and $\beta:\cP^p(M)\to M$ be a $p$-contractive
barycentric map. For every $\ffi\in L^p(\Omega;M)$ we have a
sequence $\{E_{\cB_n}^\beta(\ffi)\}_{n=1}^\infty$ of
$\beta$-conditional expectations, which we call a
\emph{$\beta$-martingale} of regular type with respect to
$\{\cB_n\}$. (A different and more intrinsic definition will be given in
Definition \ref{D-3.5}.)

A main result of this section is the martingale convergence theorem for $\beta$-martingales
of regular type. To prove this, we follow the idea of the proof of Banach's theorem given
in \cite[IV.11.3]{DS}. So we treat the space $\cM(\Omega;\bR)$ of measurable real functions
on $\Omega$, where $f=g$ in $\cM(\Omega;\bR)$ is as usual understood as
$f(\omega)=g(\omega)$ a.e. As is well-known \cite{DS}, $\cM(\Omega,\bR)$ is a
Fr\'echet space with the complete metric $\rho(f,g)=|f-g|_\bP$, where
\begin{eqnarray}\label{E:W}
|f|_\bP:=\inf_{\alpha>0}[\alpha+\bP\{\omega:|f(\omega)|>\alpha\}],\qquad f\in\cM(\Omega;\bR).
\end{eqnarray}
Note that the topology induced by $|\cdot|_\bP$ on $\cM(\Omega;\bR)$ coincides with the
topology of convergence in measure $\bP$.

\begin{thm}\label{T:martingale}
Assume that $(\Omega,\cA,\bP)$ be a standard Borel probability space. Let $\cB_n$,
$n\in\bN\cup\{\infty\}$, be sub-$\sigma$-algebras of $\cA$, either increasing or decreasing,
and let $1\le p<\infty$ and $\beta:\cP^p(M)\to M$ be as above. Then for every
$\ffi\in L^p(\Omega;M)$, as $n\to\infty$,
$$
\bd_p\bigl(E_{\cB_n}^\beta(\ffi),E_{\cB_\infty}^\beta(\ffi)\bigr)\,\longrightarrow\,0
\quad\mbox{and}\quad
d\bigl(E_{\cB_n}^\beta(\ffi)(\omega),E_{\cB_\infty}^\beta(\ffi)(\omega)\bigr)
\,\longrightarrow\,0\ \ a.e.
$$
\end{thm}

\begin{proof}
First, assume that $\ffi$ is a simple function, so
$\ffi=\sum_{j=1}^K\1_{A_j}x_j$ with $x_j\in M$ and a measurable
partition $\{A_1,\dots,A_K\}$ of $\Omega$. By \eqref{F-2.3} we can
write
$$
E_{\cB_n}^\beta(\ffi)(\omega)=\beta\Biggl(\sum_{j=1}^K\xi_{j,n}(\omega)\delta_{x_j}\Biggr),
\qquad\omega\in\Omega,\ n\in\bN\cup\{\infty\},
$$
where $\xi_{j,n}(\omega)=E_{\cB_n}(\1_{A_j})(\omega)$, the usual
conditional expectation of $\1_{A_j}$ with respect to $\cB_n$. Here
we may assume that $\xi_{j,n}(\omega)\ge0$ and
$\sum_{j=1}^K\xi_{j,n}(\omega)=1$ for all $\omega\in\Omega$ and
$n\in\bN\cup\{\infty\}$. The classical martingale convergence
theorem (see, e.g., \cite{Do}) says that
$\xi_{j,n}\to\xi_{j,\infty}$ in $L^1$-norm and for a.e.\
$\omega\in\Omega$ as $n\to\infty$. With
$\Delta:=\diam\{x_1,\dots,x_K\}$ we have by Lemma \ref{L-1.3}
\begin{align*}
d\bigl(E_{\cB_n}^\beta(\ffi)(\omega),E_{\cB_\infty}^\beta(\ffi)(\omega)\bigr)
&\le d_p^W\Biggl(\sum_{j=1}^K\xi_{j,n}(\omega)\delta_{x_j},
\sum_{j=1}^K\xi_{j,\infty}(\omega)\delta_{x_j}\Biggr) \\
&\le\Delta\Biggl[{1\over2}
\sum_{j=1}^K|\xi_{j,n}(\omega)-\xi_{j,\infty}(\omega)|\Biggr]^{1/p}
\longrightarrow\,0\ \ \mbox{a.e.}\quad\mbox{as $n\to\infty$},
\end{align*}
and
\begin{align*}
\bd_p\bigl(E_{\cB_n}^\beta(\ffi),E_{\cB_\infty}^\beta(\ffi)\bigr)
&\le\Delta\Biggl[{1\over2}\sum_{j=1}^K
\int_\Omega|\xi_{j,n}(\omega)-\xi_{j,\infty}(\omega)|\,d\bP\Biggr]^{1/p}
\longrightarrow\,0\ \ \mbox{as $n\to\infty$}.
\end{align*}

For general $\ffi\in L^p(\Omega;M)$ choose a sequence $\{\ffi_k\}$
of simple functions such that $\bd_p(\ffi,\ffi_k)\to0$. By Theorem
\ref{T-2.9}\,(2) we have
\begin{align*}
\bd_p\bigl(E_{\cB_n}^\beta(\ffi),E_{\cB_\infty}^\beta(\ffi)\bigr)
&\le\bd_p\bigl(E_{\cB_n}^\beta(\ffi),E_{\cB_n}^\beta(\ffi_k)\bigr)
+\bd_p\bigl(E_{\cB_n}^\beta(\ffi_k),E_{\cB_\infty}^\beta(\ffi_k)\bigr) \\
&\qquad+\bd_p\bigl(E_{\cB_\infty}^\beta(\ffi_k),E_{\cB_\infty}^\beta(\ffi)\bigr) \\
&\le2\bd_p(\ffi,\ffi_k)+\bd_p\bigl(E_{\cB_n}^\beta(\ffi_k),
E_{\cB_\infty}^\beta(\ffi_k)\bigr)
\end{align*}
so that
$$
\limsup_{n\to\infty}\bd_p\bigl(E_{\cB_n}^\beta(\ffi),E_{\cB_\infty}^\beta(\ffi)\bigr)
\le2\bd_p(\ffi,\ffi_k)\,\longrightarrow\,0\quad\mbox{as $k\to\infty$}.
$$
Hence $\bd_p\bigl(E_{\cB_n}^\beta(\ffi),E_{\cB_\infty}^\beta(\ffi)\bigr)\to0$ as $n\to\infty$.

It remains to prove the a.e.\ convergence. Choose an $x_0\in M$ and
let $\ffi_0:=\1_\Omega x_0$. Let
$(\bP_\omega^{(n)})_{\omega\in\Omega}$ be a disintegration of $\bP$
with respect to $\cB_n$ (see Section 2). Since
$x_0=\beta((\ffi_0)_*\bP_\omega^{(n)})$ for all $\omega\in\Omega$
and $n\in\bN\cup\{\infty\}$, note that
\begin{align*}
&d\bigl(E_{\cB_m}^\beta(\ffi)(\omega),E_{\cB_n}^\beta(\ffi)(\omega)\bigr) \\
&\qquad=d\bigl(\beta(\ffi_*\bP_\omega^{(m)}),\beta((\ffi_0)_*\bP_\omega^{(m)})\bigr)
+d\bigl(\beta(\ffi_*\bP_\omega^{(n)}),\beta((\ffi_0)_*\bP_\omega^{(n)})\bigr) \\
&\qquad\le\biggl[\int_\Omega
d^p(\ffi(\tau),x_0)\,d\bP_\omega^{(m)}(\tau)\biggr]^{1/p}
+\biggl[\int_\Omega
d^p(\ffi(\tau),x_0)\,d\bP_\omega^{(n)}(\tau)\biggr]^{1/p}
\end{align*}
for a.e. $\omega\in \Omega$, where we have used Lemma \ref{L-2.7}\,(2). Therefore, we find that
$$
\sup_{m,n\ge1}d\bigl(E_{\cB_m}^\beta(\ffi)(\omega),E_{\cB_n}^\beta(\ffi)(\omega)\bigr)
\le2\biggl[\sup_{n\ge1}E_{\cB_n}[d^p(\ffi,x_0)](\omega)\biggr]^{1/p}
\ \ \mbox{a.e.}\ \omega,
$$
where $\bigl\{E_{\cB_n}[d^p(\ffi,x_0)]\bigr\}_{n=1}^\infty$ is the usual martingale for
the function $\omega\mapsto d^p(\ffi(\omega),x_0)$ in $L^1(\Omega;\bR)$. For each
$\ffi\in L^p(\Omega;M)$, since the classical a.e.\ martingale convergence gives
$$
\sup_{n\ge1}E_{\cB_n}[d^p(\ffi,x_0)](\omega)<\infty
\ \ \mbox{a.e.}\ \omega,
$$
we can define a function $W(\ffi)\in\cM(\Omega;\bR)$ by
$$
W(\ffi)(\omega):=\lim_{k\to\infty}\sup_{m,n\ge k}
d\bigl(E_{\cB_m}^\beta(\ffi)(\omega),E_{\cB_n}^\beta(\ffi)(\omega)\bigr)
\ \ \mbox{a.e.}\ \omega.
$$
Then it is obvious that $\lim_{n\to\infty}E_{\cB_n}^\beta(\ffi)(\omega)$ exists a.e.\ if and
only if $W(\ffi)=0$ as an element of $\cM(\Omega;\bR)$. In this case, the a.e.\ limit of
$E_{\cB_n}^\beta(\ffi)(\omega)$ must be $E_{\cB_\infty}^\beta(\ffi)(\omega)$ for
a.e.\ $\omega$ since $E_{\cB_n}^\beta(\ffi)\to E_{\cB_\infty}^\beta(\ffi)$ in $L^p$ sense
as already shown above. Furthermore, we have shown above that
$\lim_{n\to\infty}E_{\cB_n}^\beta(\ffi)(\omega)$ exists a.e.\ if $\ffi$ is a simple function.
From the denseness of the simple functions in $L^p(\Omega;M)$, it suffices to prove that $W$
is a continuous map from $L^p(\Omega;M)$ into $\cM(\Omega;\bR)$ equipped with topology of
convergence in measure $\bP$.

To prove the last statement, note that for every $\ffi,\psi\in
L^p(\Omega;M)$ and every $m,n\ge1$,
\begin{align*}
d\bigl(E_{\cB_m}^\beta(\ffi)(\omega),E_{\cB_n}^\beta(\ffi)(\omega)\bigr)
 &\le
d\bigl(E_{\cB_m}^\beta(\ffi)(\omega),E_{\cB_m}^\beta(\psi)(\omega)\bigr)
+d\bigl(E_{\cB_m}^\beta(\psi)(\omega),E_{\cB_n}^\beta(\psi)(\omega)\bigr) \\
&\qquad
+d\bigl(E_{\cB_n}^\beta(\psi)(\omega),E_{\cB_n}^\beta(\ffi)(\omega)\bigr),
\end{align*}
which implies that
\begin{align*}
&\sup_{m,n\ge k}d\bigl(E_{\cB_m}^\beta(\ffi)(\omega),E_{\cB_n}^\beta(\ffi)(\omega)\bigr) \\
&\quad\le\sup_{m,n\ge k}
d\bigl(E_{\cB_m}^\beta(\psi)(\omega),E_{\cB_n}^\beta(\psi)(\omega)\bigr)
+2\sup_{n\ge1}d\bigl(E_{\cB_n}^\beta(\ffi)(\omega),E_{\cB_n}^\beta(\psi)(\omega)\bigr).
\end{align*}
We hence have
\begin{align*}
|W(\ffi)(\omega)-W(\psi)(\omega)|
&\le2\sup_{n\ge1}d\bigl(E_{\cB_n}^\beta(\ffi)(\omega),E_{\cB_n}^\beta(\psi)(\omega)\bigr) \\
&\le2\sup_{n\ge1}\biggl[\int_\Omega d^p(\ffi(\tau),\psi(\tau))\,d\bP_\omega^{(n)}(\tau)
\biggr]^{1/p}\ \ \mbox{a.e.}\ \omega
\end{align*}
by Lemma \ref{L-2.7}\,(2). Therefore,
\begin{align}\label{F3.1}
|W(\ffi)(\omega)-W(\psi)(\omega)|
\le2\biggl[\sup_{n\ge1}E_{\cB_n}[d^p(\ffi,\psi)](\omega)\biggr]^{1/p},
\end{align}
where $\bigl\{E_{\cB_n}[d^p(\ffi,\psi)]\bigr\}_{n=1}^\infty$ is the usual martingale for the
function $\omega\mapsto d^p(\ffi(\omega),\psi(\omega))$ in $L^1(\Omega;\bR)$. Note that the
function $\omega\mapsto\sup_{n\ge1}E_{\cB_n}[d^p(\ffi,\psi)](\omega)$ belongs to
$\cM(\Omega;\bR)$
since this supremum is finite for a.e.\ $\omega$. From the proof of \cite[IV.11.3]{DS} it
follows that
\begin{align}\label{F3.2}
f\in L^1(\Omega;\bR)\ \longmapsto\ \sup_{n\ge1}|E_{\cB_n}(f)(\omega)|\in\cM(\Omega;\bR)
\end{align}
is continuous at $f=0$. If $\ffi,\ffi_k\in L^p(\Omega;M)$ and $\bd_p(\ffi,\ffi_k)\to0$, then
$d^p(\ffi,\ffi_k)\to0$ in $L^1$-norm, and from \eqref{F3.1} and the continuity of \eqref{F3.2},
we obtain $W(\ffi_k)\to W(\ffi)$ in $\cM(\Omega;\bR)$, as desired.
\end{proof}

Sturm \cite{St02} showed a convergence theorem for martingales with locally compact range in a
global NPC space, where martingales were introduced from the viewpoint of stochastic processes
differently from those discussed above. In the rest of this section we consider Sturm's type
martingales in our general setting.

Assume that $(\Omega,\cA,\bP)$ and $\beta:\cP^p(M)\to M$ are as in Theorem \ref{T:martingale},
and let $\cB_n$, $n\in\bN\cup\{\infty\}$, be an increasing sequence of sub-$\sigma$-algebras
of $\cA$. Following \cite{St02}, for $\ffi\in L^p(\Omega;M)$ and $m\ge k\ge1$, we define
$$
E^\beta\bigl[\ffi\|(\cB_n)_{m\ge n\ge k}\bigr]
:=E_{\cB_k}^\beta\circ E_{\cB_{k+1}}^\beta\circ\dots\circ E_{\cB_m}^\beta(\ffi),
$$
which is an element of $L^p(\Omega,\cB_k,\bP;M)$. The proof of the next lemma is
based on Theorem \ref{T:martingale}.

\begin{lemma}\label{L-3.2}
For every $\ffi\in L^p(\Omega;M)$ and $k\ge1$ the following equal
limits exist$:$
\begin{align}\label{F-3.4}
\lim_{m\to\infty}E^\beta\bigl[\ffi\|(\cB_n)_{m\ge n\ge k}\bigr]
=\lim_{m\to\infty}E^\beta\bigl[E_{\cB_\infty}^\beta(\ffi)\|(\cB_n)_{m\ge n\ge k}\bigr]
\quad\mbox{in metric $\bd_p$}.
\end{align}
\end{lemma}

\begin{proof}
For notational simplicity, for any $\ffi\in L^p(\Omega;M)$ write
$\ffi_\infty:=E_{\cB_\infty}^\beta(\ffi)$ and
$\ffi_{m,k}:=E^\beta\bigl[\ffi\|(\cB_n)_{m\ge n\ge k}\bigr]$ for $m\ge k\ge1$.
For $l>m\ge k$ we have
\begin{align}\label{ee}
\bd_p(\ffi_{m,k},\ffi_{l,k})
\le\bd_p(\ffi_{m,k},(\ffi_\infty)_{m,k})+\bd_p((\ffi_\infty)_{m,k},(\ffi_\infty)_{l,k})
+\bd_p((\ffi_\infty)_{l,k},\ffi_{l,k}).
\end{align}
Moreover, by Theorem \ref{T-2.9}\,(2) and Theorem
3.1,
\begin{align}
\bd_p(\ffi_{m,k},(\ffi_\infty)_{m,k})
&\le\bd_p\bigl(E_{\cB_m}^\beta(\ffi),E_{\cB_m}^\beta(\ffi_\infty)\bigr) \nonumber\\
&\le\bd_p\bigl(E_{\cB_m}^\beta(\ffi),\ffi_\infty)
+\bd_p\bigl(\ffi_\infty,E_{\cB_m}^\beta(\ffi_\infty)\bigr) \nonumber\\
&\longrightarrow\,0\quad\mbox{as $m\to\infty$}, \label{rf}
\end{align}
and similarly $\bd_p(\ffi_{l,k},(\ffi_\infty)_{l,k})\to0$ as $l\to\infty$. Since
$\bigcup_{n=1}^\infty L^p(\Omega,\cB_n,\bP;M)$ is $\bd_p$-dense in
$L^p(\Omega,\cB_\infty,\bP;M)$, for every $\eps>0$ one can choose a
$\psi\in L^p(\Omega,\cB_n,\bP;M)$ for some $n\ge1$ such that $\bd_p(\psi,\ffi_\infty)<\eps$.
One has
\begin{align*}
\bd_p((\ffi_\infty)_{m,k},(\ffi_\infty)_{l,k})
&=\bd_p\bigl(E_{\cB_k}^\beta\circ\dots\circ E_{\cB_m}^\beta(\ffi_\infty),
E_{\cB_k}^\beta\circ\dots\circ E_{\cB_l}^\beta(\ffi_\infty)\bigr) \\
&\le\bd_p\bigl(\ffi_\infty,
E_{\cB_{m+1}}^\beta\circ\dots\circ E_{\cB_l}^\beta(\ffi_\infty)\bigr) \\
&\le\bd_p(\ffi_\infty,\psi)
+\bd_p\bigl(\psi,E_{\cB_{m+1}}^\beta\circ\dots\circ E_{\cB_l}^\beta(\ffi_\infty)\bigr).
\end{align*}
For $l>m\ge\max\{k,n\}$, since $\psi=E_{\cB_{m+1}}^\beta\circ\dots\circ
E_{\cB_l}^\beta(\psi)$, one has
$\bd_p((\ffi_\infty)_{m,k},(\ffi_\infty)_{l,k})\le2\eps$, which
implies that $\bd_p((\ffi_\infty)_{m,k},(\ffi_\infty)_{l,k})\to0$ as
$l,m\to\infty$. Hence it follows from (\ref{ee}) that $\bd_p(\ffi_{m,k},\ffi_{l,k})\to0$
as $l,m\to\infty$, so that $\ffi_{m,k}$ converges in $\bd_p$ to some element of
$L^p(\Omega,\cB_k,\bP;M)$ as $m\to\infty$. Thanks to (\ref{rf}), $(\ffi_\infty)_{m,k}$ also
converges to the same limit as $m\to\infty$.
\end{proof}

\begin{definition}\label{D-3.3}\rm
For every $\ffi\in L^p(\Omega;M)$ and $k\ge1$, we write
$$
E^\beta\bigl[\ffi\|(\cB_n)_{n\ge k}\bigr]
$$
for the equal limits in \eqref{F-3.4}, which is an element of $L^p(\Omega,\cB_k,\bP;M)$ and
we call the \emph{filtered $\beta$-conditional expectation} of $\ffi$ with respect to
$(\cB_n)_{n\ge k}$.
\end{definition}

The associativity in (4) below is a merit of filtered $\beta$-conditional expectations,
which is not satisfied for those in Theorem \ref{T:martingale} (see Example \ref{E-2.12}).

\begin{prop}\label{P-3.4}
Let $\ffi,\psi\in L^p(\Omega;M)$.
\begin{itemize}
\item[(1)] $E^\beta\bigl[\ffi\|(\cB_n)_{n\ge k}\bigr]\in L^p(\Omega,\cB_k,\bP;M)$ for all
$k\ge1$.
\item[(2)] For every $k\ge1$, $\ffi\in L^p(\Omega,\cB_k,\bP;M)$ if and only if
$E^\beta\bigl[\ffi\|(\cB_n)_{n\ge k}\bigr]=\ffi$.
\item[(3)] $\bd_p\bigl(E^\beta\bigl[\ffi\|(\cB_n)_{n\ge k}\bigr],
E^\beta\bigl[\psi\|(\cB_n)_{n\ge k}\bigr]\bigr)\le\bd_p(\ffi,\psi)$ for all $k\ge1$.
\item[(4)] For every $l\ge k\ge1$,
$$
E^\beta\bigl[E^\beta\bigl[\ffi\|(\cB_n)_{n\ge l}\bigr]\|(\cB_n)_{n\ge k}\bigr]
=E^\beta\bigl[\ffi\|(\cB_n)_{n\ge k}\big].
$$
\end{itemize}
\end{prop}

\begin{proof}
(1) is obvious.

(2)\enspace
If $\ffi\in L^p(\Omega,\cB_k,\bP;M)$, then
$E_{\cB_k}^\beta\circ\dots\circ E_{\cB_m}^\beta(\ffi)=\ffi$ for all $m\ge k$ and hence
$E^\beta\bigl[\ffi\|(\cB_n)_{n\ge k}\bigr]=\ffi$. The converse is obvious from (1).

(3)\enspace
For every $m\ge k\ge1$ we have by Theorem \ref{T-2.9}\,(2)
$$
\bd_p\bigl(E_{\cB_k}^\beta\circ\dots\circ E_{\cB_m}^\beta(\ffi),
E_{\cB_k}^\beta\circ\dots\circ E_{\cB_m}^\beta(\psi)\bigr)
\le\bd_p(\ffi,\psi),
$$
whose limit as $m\to\infty$ is the asserted inequality.

(4)\enspace
For simplicity, for $\ffi\in L^p(\Omega;M)$ write
$\ffi_{\infty,k}:=E^\beta\bigl[\ffi\|(\cB_n)_{n\ge k}\bigr]$ for $k\ge1$.
Since $\ffi_{\infty,l}\in L^p(\Omega,\cB_l,\bP;M)$ by (1), we have for every $m\ge l>k$
\begin{align*}
E_{\cB_k}^\beta\circ\dots\circ E_{\cB_m}^\beta(\ffi_{\infty,l})
&=E_{\cB_k}^\beta\circ\dots\circ E_{\cB_{l-1}}^\beta(\ffi_{\infty,l}) \\
&=\lim_{m\to\infty}E_{\cB_k}^\beta\circ\dots\circ E_{\cB_{l-1}}^\beta
\circ E_{\cB_l}^\beta\circ\dots\circ E_{\cB_m}^\beta(\ffi)=\ffi_{\infty,k}.
\end{align*}
Therefore,
$$
E^\beta\bigl[\ffi_{\infty,l}\|(\cB_n)_{n\ge k}\bigr]
=\lim_{m\to\infty}E_{\cB_k}^\beta\circ\dots\circ E_{\cB_m}^\beta(\ffi_{\infty,l})
=\ffi_{\infty,k},
$$
as required.
\end{proof}

Following \cite[Definition 4.1]{St02} we define:

\begin{definition}\label{D-3.5}\rm
A sequence $\{\ffi_k\}_{k=1}^\infty$ in $L^p(\Omega;M)$ is called a
\emph{filtered $\beta$-martingale} with respect to $\{\cB_n\}_{n=1}^\infty$ if
$\ffi_k\in L^p(\Omega,\cB_k,\bP;M)$ for every $k\ge1$
and
\begin{align}\label{F-3.7}
E^\beta\bigl[\ffi_{k+1}\|(\cB_n)_{n\ge k}\bigr]=\ffi_k,\qquad k\ge1.
\end{align}
\end{definition}

By associativity in Proposition \ref{P-3.4}\,(4), property \eqref{F-3.7} is equivalent to
$$
E^\beta\bigl[\ffi_l\|(\cB_n)_{n\ge k}\bigr]=\ffi_k,\qquad l\ge k\ge1.
$$
For any $\ffi\in L^p(\Omega;M)$, it is clear that the sequence
$\ffi_k:=E^\beta\bigl[\ffi\|(\cB_n)_{n\ge k}\bigr]$, $k\ge1$, is a filtered $\beta$-martingale
with respect to $\{\cB_n\}$. The next theorem includes its $\bd_p$-convergence.

\begin{thm}\label{T-3.6}
Let $\{\ffi_k\}_{k=1}^\infty$ be a filtered $\beta$-martingale with
respect to $\{\cB_n\}$. Then the following are equivalent$:$
\begin{itemize}
\item[(i)] there exists a $\ffi\in L^p(\Omega;M)$ such that
$\ffi_k=E^\beta\bigl[\ffi\|(\cB_n)_{n\ge k}\bigr]$ for all $k\ge1;$
\item[(ii)] $\ffi_k$ converges to some $\ffi_\infty\in L^p(\Omega,\cB_\infty,\bP;M)$ in
metric $\bd_p$ as $k\to\infty$.
\end{itemize}
\end{thm}

\begin{proof}
(i)\,$\implies$\,(ii).\enspace
Let $\ffi\in L^p(\Omega;M)$ be as stated in (i). Let
$\ffi_\infty:=E_{\cB_\infty}^\beta(\ffi)\in L^p(\Omega,\cB_\infty,\bP;M)$. For every $\eps>0$,
by Theorem \ref{T:martingale} one can choose a $\psi\in L^p(\Omega,\cB_l,\bP;M)$ for some
$l\ge1$ such that $\bd_p(\psi,\ffi_\infty)<\eps$. For every $k\ge l$, since
$E^\beta\bigl[\psi\|(\cB_n)_{n\ge k}\bigr]=\psi$ by Proposition \ref{P-3.4}\,(2), one has
\begin{align*}
\bd_p(\ffi_k,\ffi_\infty)
&\le\bd_p\bigl(E^\beta\bigl[\ffi_\infty\|(\cB_n)_{n\ge k}\bigr],
E\bigl[\psi\|(\cB_n)_{n\ge k}\bigr]\bigr)+\bd_p(\psi,\ffi_\infty) \\
&\le2\bd_p(\psi,\ffi_\infty)<2\eps
\end{align*}
by Proposition \ref{P-3.4}\,(3). Hence (ii) follows.

(ii)\,$\implies$\,(i).\enspace
For $l\ge k\ge1$ one has by Proposition \ref{P-3.4} again
\begin{align*}
\bd_p\bigl(\ffi_k,E^\beta\bigl[\ffi_\infty\|(\cB_n)_{n\ge k}\bigr]\bigr)
&=\bd_p\bigl(E^\beta\bigl[\ffi_l\|(\cB_n)_{n\ge k}\bigr],
E^\beta\bigl[\ffi_\infty\|(\cB_n)_{n\ge k}\bigr]\bigr) \\
&\le\bd_p(\ffi_l,\ffi_\infty)\,\longrightarrow\,0\ \ \mbox{as $l\to\infty$}.
\end{align*}
Hence $\ffi_k=E^\beta\bigl[\ffi_\infty\|(\cB_n)_{n\ge k}\bigr]$.
\end{proof}

\begin{remark}\label{R-3.6}\rm
When $(M,d)$ is a locally compact global NPC space and $\beta$ is a
canonical barycentric map $\lambda$, it follows from \cite[Theorem
4.11]{St02} (and Theorem \ref{T-2.13}) that if $\{\ffi_k\}$ in
$L^p(\Omega;M)$ is a filtered martingale and
$\sup_k\bd_p(z,\ffi_k)<\infty$ for some $z\in M$, then there exists
a $\cB_\infty$-measurable function $\ffi_\infty:\Omega\to M$ such
that $\ffi_k(\omega)\to\ffi_\infty(\omega)$ $\bP$-a.e. From the
Hopf-Rinow theorem (cf.\ \cite{BH}) we see that this result holds
more generally when $(M,d)$ is a locally compact and complete length
space and $\beta:\cP^p(M)\to M$ is any contractive barycentric map.
But it does not seem easy to extend the $\bP$-a.e.\ martingale
convergence of filtered $\beta$-martingales to our general setting.
Although the details are omitted here, the same result holds under
an even more general situation that $(M,d)$ satisfies
\emph{finite-compactness} with respect to $\beta$ in the sense that for any
finite set $Q_0$ in $M$ the closure of $\bigcup_{n=1}^\infty Q_n$ is
compact, where
$$
Q_n:=\bigl\{\beta(\mu):\mu\in\cP_0(M),\,\supp(\mu)\subset Q_{n-1}\bigr\},
\qquad n\in\bN.
$$
This finite-compactness property clearly holds in the case of Banach spaces
with the arithmetic mean map. But it is unknown whether it holds in the case
where $M=\bbP(\cH)$ on an infinite-dimensional Hilbert space $\cH$ and $\beta$
is the Karcher barycenter $G$ (see Example \ref{E-4.5}\,(b) below).
\end{remark}

\section{Ergodic theorem}

Let $T$ be a $\bP$-preserving measurable transformation on $(\Omega,\cA,\bP)$. It is clear
that the map $\ffi\mapsto\ffi\circ T$ is a $\bd_p$-isometric transformation on $L^p(\Omega;M)$.
(Although we may treat a measure-preserving action of an amenable group $G$ as in \cite{Na},
we consider the case $G=\bZ$ for the sake of simplicity.)

Let $1\le p<\infty$ and $\beta:\cP^p(M)\to M$ be a
$p$-contractive barycentric map. For each $\ffi\in
L^p(\Omega;M)$ we define the {\it empirical measures} (random
probability measures) of $\ffi$ as
$$
\mu_n^\ffi(\omega):={1\over n}\sum_{k=0}^{n-1}\delta_{\ffi(T^k\omega)},\qquad n\in\bN,
$$
i.e., for Borel sets $B\subset M$,
$$
\mu_n^\ffi(\omega)(B)={\#\{k\in\{0,1,\dots,n-1\}:\ffi(T^k\omega)\in B\}\over n},
$$
and consider the sequence of $M$-valued functions
$\beta(\mu_n^\ffi)(\omega):=\beta(\mu_n^\ffi(\omega))$, $\omega\in\Omega$, for $n\in\bN$.

\begin{lemma}\label{L-4.1}
For every $\ffi,\psi\in L^p(\Omega;M)$ we have $\beta(\mu_n^\ffi)\in L^p(\Omega;M)$ and
$$
\bd_p(\beta(\mu_n^\ffi),\beta(\mu_n^\psi))\le\bd_p(\ffi,\psi),\qquad n\in\bN.
$$
\end{lemma}

\begin{proof}
Let $n\in\bN$ be arbitrarily fixed. First, assume that $\ffi$ is a
simple function so that $\ffi=\sum_{i=1}^K\1_{A_i}\delta_{x_i}$,
where $x_1,\dots,x_K\in M$ and ${\mathcal{F}}=\{A_1,\dots,A_K\}$ is
a measurable partition of $\Omega$. Then, as easily seen, we can write
$$
\mu_n^\ffi(\omega)=\sum_{A\in\bigvee_{k=0}^{n-1}T^{-k}{\mathcal{F}}}\1_A(\omega)\mu_A
$$
with $\mu_A\in\cP_0(M)$, where $\bigvee_{k=0}^{n-1}T^{-k}{\mathcal{F}}$ is the finite
partition generated by $T^{-k}\mathcal{F}$, $0\le k\le n-1$. Therefore,
$$
\beta(\mu_n^\ffi(\omega))=\sum_{A\in\bigvee_{k=0}^{n-1}
T^{-k}{\mathcal{F}}}\1_A(\omega)\beta(\mu_A)
$$
so that $\omega\mapsto\beta(\mu_n^\ffi(\omega))$ is a simple function.

Next, let $\ffi,\psi$ be arbitrary elements in $L^p(\Omega;M)$. For each fixed $\omega\in\Omega$,
since
$$
\pi:={1\over n}\sum_{k=0}^{n-1}\delta_{(\ffi(T^k\omega),\psi(T^k\omega))}\in\cP(M\times M)
$$
is in $\Pi(\mu_n^\ffi(\omega),\mu_n^\psi(\omega))$, we have
$$
d_p^W(\mu_n^\ffi(\omega),\mu_n^\psi(\omega))
\le\Biggl[{1\over n}\sum_{k=0}^{n-1}d^p(\ffi(T^k\omega),\psi(T^k\omega))\Biggr]^{1/p}.
$$
From this and the $p$-contractivity of $\beta$ we find that
\begin{equation}\label{F-4.2}
d(\beta(\mu_n^\ffi(\omega)),\beta(\mu_n^\psi(\omega)))
\le\Biggl[{1\over n}\sum_{k=0}^{n-1}d^p(\ffi(T^k\omega),\psi(T^k\omega))\Biggr]^{1/p},
\qquad\omega\in\Omega.
\end{equation}
Now, choose $M$-valued simple functions $\ffi_l$ ($l\in\bN$) such that
$d(\ffi(\omega),\ffi_l(\omega))\to0$ a.e.\ as $l\to\infty$. Letting $\psi=\ffi_l$ in
\eqref{F-4.2} we have $d(\beta(\mu_n^\ffi(\omega)),\beta(\mu_n^{\ffi_l}(\omega)))\to0$
a.e.\ as $l\to\infty$. Since $\beta(\mu_n^{\ffi_l})$'s are simple functions as proved
above, it follows that $\beta(\mu_n^\ffi)$ is a strongly measurable function on $\Omega$.
Letting $\psi=\1_\Omega x$ ($x\in M$) in \eqref{F-4.2}, since
$\beta(\mu_n^\psi)(\omega)=x$ for all $\omega\in\Omega$, we have
\begin{align*}
\int_\Omega d^p(\beta(\mu_n^\ffi(\omega)),x)\,d\bP(\omega)
&\le\int_\Omega{1\over
n}\sum_{k=0}^{n-1}d^p(\ffi(T^k\omega),x)\,d\bP(\omega)=\int_\Omega
d^p(\ffi(\omega),x)\,d\bP(\omega)<\infty
\end{align*}
so that $\beta(\mu_n^\ffi)\in L^p(\Omega;M)$. Finally, it follows from \eqref{F-4.2} again that
\begin{align*}
\int_\Omega d^p(\beta(\mu_n^\ffi(\omega)),\beta(\mu_n^\psi(\omega)))\,d\bP(\omega)
&\le\int_\Omega{1\over n}\sum_{k=0}^{n-1}d^p(\ffi(T^k\omega),\psi(T^k\omega))\,d\bP(\omega) \\
&=\int_\Omega d^p(\ffi(\omega),\psi(\omega))\,d\bP(\omega)
\end{align*}
so that $\bd_p(\beta(\mu_n^\ffi),\beta(\mu_n^\psi))\le\bd_p(\ffi,\psi)$.
\end{proof}

In \cite{Au} Austin obtained an $L^2$-ergodic theorem for the canonical barycentric
map on a global NPC space. In \cite{Na} Navas established an $L^1$-ergodic theorem for
a specific contractive barycentric map on a metric space of nonpositive curvature
in the sense of Busemann (a weaker notion than that of a global NPC space). In \cite{Li},
Navas' ergodic theorem was proved for parametrized barycentric maps extending the Cartan
(or Karcher) barycenter on the positive definite matrices.

In this section we give an $L^p$-ergodic theorem for $1\leq
p<\infty$ on a general complete metric space with a general
$p$-contractive barycentric map $\beta$. Moreover, we give the
description of the ergodic limit function in terms of the $\beta$-conditional
expectation. Since the proof of the next theorem is along the
essentially same lines as \cite{Au,Na}, we shall only present its
sketchy version.

\begin{thm}\label{T-4.2}
Let $1\le p<\infty$ and $\beta:\cP^p(M)\to M$ be a $p$-contractive barycentric map. Then there
exists a map $\Gamma$ from $L^p(\Omega;M)$ onto
$\{\ffi\in L^p(\Omega;M):\ffi\circ T=\ffi\}$ such that for every $\ffi,\psi\in L^p(\Omega;M)$,
\begin{itemize}
\item[(i)] $d(\beta(\mu_n^\ffi(\omega)),\Gamma(\ffi)(\omega))\to0$ a.e.\ as $n\to\infty$,
\item[(ii)] $\bd_p(\beta(\mu_n^\ffi),\Gamma(\ffi))\to0$ as $n\to\infty$,
\item[(iii)] $\bd_p(\Gamma(\ffi),\Gamma(\psi))\le\bd_p(\ffi,\psi)$.
\end{itemize}
Furthermore, if $T$ is ergodic, then $\Gamma(\ffi)$ is constant with value $E^\beta(\ffi)$,
the $\beta$-expectation of $\ffi$ $($see Definition $\ref{D-2.1})$.
\end{thm}

\begin{proof}
Let $\ffi,\psi\in L^p(\Omega;M)$. Applying the maximal ergodic theorem to the function
$f(\omega):=d^p(\ffi(\omega),\psi(\omega))\in L^1(\Omega;\bR)$ we have for every $\lambda>0$,
$$
\bP\Biggl\{\omega\in\Omega:\sup_{n\ge1}{1\over n}\sum_{k=0}^{n-1}
d^p(\ffi(T^k\omega),\psi(T^k\omega))>\lambda\Biggr\}
\le{1\over\lambda}\,\bd_p^p(\ffi,\psi),
$$
which together with \eqref{F-4.2} implies that
\begin{equation}\label{F-4.3}
\bP\biggl\{\omega\in\Omega:\sup_{n\ge1}
d^p(\beta(\mu_n^\ffi(\omega)),\beta(\mu_n^\psi(\omega)))>\lambda\biggr\}
\le{1\over\lambda}\,\bd_p^p(\ffi,\psi).
\end{equation}

Now, assume that $\psi$ is a simple function so that $\psi=\sum_{i=1}^K\1_{A_i}x_i$ with
$x_i\in M$ and a measurable partition $\{A_1,\dots,A_K\}$ of $\Omega$. Note that we can
write
\begin{align}\label{F-4.4}
\mu_n^\psi(\omega)=\sum_{i=1}^K\Biggl({1\over n}\sum_{k=0}^{n-1}
\1_{A_i}(T^k\omega)\Biggr)\delta_{x_i}.
\end{align}
By the usual Birkhoff ergodic theorem, there are $T$-invariant functions $\xi_i\in
L^1(\Omega;{\Bbb R})$ ($1\le i\le K$) such that
$\frac{1}{n}\sum_{k=0}^{n-1}\1_{A_i}(T^k\omega)\to\xi_i(\omega)$ a.e.\ as $n\to\infty$.
Therefore, with $\mu_\infty^\psi(\omega):=\sum_{i=1}^K\xi_i(\omega)\delta_{x_i}$ we have
\begin{equation}\label{F-4.5}
d(\beta(\mu_n^\psi(\omega)),\beta(\mu_\infty^\psi(\omega)))
\le d_p^W(\mu_n^\psi(\omega),\mu_\infty^\psi(\omega))\longrightarrow0
\ \ \mbox{a.e.\quad as $n\to\infty$}
\end{equation}
from Lemma \ref{L-1.3}. Choose simple functions $\psi_l$ ($l\in\bN$) such that
$d(\ffi(\omega),\psi_l(\omega))\to0$ a.e.\ and $\bd_p(\ffi,\psi_l)\to0$ as $l\to\infty$.
For every $\eps\in(0,1)$ choose a $\psi_l$ such that $\bd_p(\ffi,\psi_l)<\eps^2$.
Furthermore, from \eqref{F-4.5} with $\psi=\psi_l$, one can choose an
$n_\eps\in\bN$ and an $N_\eps\in\cA$ such that $\bP(N_\eps)<\eps$ and
$d(\beta(\mu_n^{\psi_l}(\omega)),\beta(\mu_\infty^{\psi_l}(\omega))\le\eps$ for
$\omega\in\Omega\setminus N_\eps$. Let
$$
\widetilde N_\eps:=\biggl\{\omega\in\Omega:\sup_{n\ge1}
d(\beta(\mu_n^\ffi(\omega)),\beta(\mu_n^{\psi_l}(\omega)))>\eps^p\biggr\}
\cup N_\eps.
$$
Then, from \eqref{F-4.3} with $\psi=\psi_l$ and $\lambda=\eps^p$, one has
$\bP(\widetilde N_\eps)<2\eps$ and for every $\omega\in\Omega\setminus\widetilde N_\eps$
and $n\ge n_\eps$, $d(\beta(\mu_n^\ffi(\omega)),\beta(\mu_\infty^{\psi_l}(\omega)))\le2\eps$,
so that
$$
d(\beta(\mu_n^\ffi(\omega)),\beta(\mu_m^\ffi(\omega)))\le4\eps,
\qquad\omega\in\Omega\setminus\widetilde N_\eps,\ n,m\ge n_\eps.
$$
Letting $\eps_k:=k^{-2}$ and $N:=\limsup_{k\to\infty}\widetilde N_{\eps_k}$, one has
$\bP(N)=0$ by the Borel-Cantelli lemma and
$$
\lim_{n,m\to\infty}d(\beta(\mu_n^\ffi(\omega)),\beta(\mu_m^\ffi(\omega)))=0,
\qquad\omega\in\Omega\setminus N,
$$
which implies that there exists a strongly measurable function
$\Gamma(\ffi):\Omega\to M$ for which property (i) holds, though
$\Gamma(\ffi)\in L^p(\Omega;M)$ as well as $\Gamma(\ffi)\circ
T=\Gamma(\ffi)$ will be proved below.

To prove (ii), note by Lemma \ref{L-4.1} that
\begin{align}
\bd_p(\beta(\mu_n^\ffi),\beta(\mu_m^\ffi))
\le2\bd_p(\ffi,\psi_l)+\bd_p(\beta(\mu_n^{\psi_l}),\beta(\mu_m^{\psi_l})). \label{F-4.6}
\end{align}
For every $\eps>0$ choose a $\psi_l$ such that $\bd_p(\ffi,\psi_l)<\eps$, and write
$\psi_l=\sum_{j=1}^K\1_{A_j}x_j$. Then it follows from \eqref{F-4.4} and Lemma \ref{L-1.3}
that with $\Delta:=\diam\{x_1,\dots,x_K\}$,
$$
\bd_p^p(\beta(\mu_n^{\psi_l}),\beta(\mu_m^{\psi_l}))
\le\Delta^p\sum_{i=1}^K\Bigg\|{1\over n}\sum_{k=0}^{n-1}\1_{A_i}\circ T^k
-{1\over m}\sum_{k=0}^{m-1}\1_{A_i}\circ T^k\Bigg\|_1.
$$
Hence, by the usual mean ergodic theorem, one can choose an $n_\eps\in\bN$ such that
$\bd_p(\beta(\mu_n^{\ffi_l}),\beta(\mu_m^{\ffi_l}))<\eps$ for all $n,m\ge n_\eps$. Therefore,
by \eqref{F-4.6}, $\bd_p(\beta(\mu_n^\ffi),\beta(\mu_m^\ffi))<3\eps$ for all $n,m\ge n_\eps$.
By property (i) and Fatou's lemma, we have
\begin{align*}
\int_\Omega d^p(\beta(\mu_n^\ffi(\omega)),\Gamma(\ffi)(\omega))\,d\bP(\omega)
\le(3\eps)^p,\qquad n\ge n_\eps,
\end{align*}
which implies that $\Gamma(\ffi)\in L^p(\Omega;M)$ and property (ii) holds.

When $\ffi,\psi\in L^p(\Omega;M)$, property (iii) follows from Lemma \ref{L-4.1} and
Fatou's lemma since property (i) implies that
$d(\beta(\mu_n^\ffi(\omega)),\beta(\mu_n^\psi(\omega)))\to
d(\Gamma(\ffi)(\omega),\Gamma(\psi)(\omega))$ a.e.\ as $n\to\infty$.

Next, we confirm that $\Gamma(\ffi)$ is $T$-invariant. For a simple function
$\psi$ it follows from \eqref{F-4.4} and Lemma \ref{L-1.3} that
$\Gamma(\psi)(T\omega)=\Gamma(\psi)(\omega)$ a.e. For an arbitrary
$\ffi\in L^p(\Omega;M)$, choose a sequence $\psi_l$ as above. Since
$\Gamma(\psi_l)\circ T=\Gamma(\psi_l)$ as verified just above, we have
$$
\bd_p(\Gamma(\ffi),\Gamma(\ffi)\circ T)
\le2\bd_p(\Gamma(\ffi),\Gamma(\psi_l))
\le2\bd_p(\ffi,\psi_l)\longrightarrow0\quad\mbox{as $l\to\infty$}
$$
thanks to property (iii). Hence $\Gamma(\ffi)=\Gamma(\ffi)\circ T$, as desired.

Finally, assume that $T$ is ergodic. For a simple function $\psi=\sum_{i=1}^K\1_{A_i}x_i$,
since $\frac{1}{n}\sum_{k=0}^{n-1}\1_{A_i}(T^k\omega)\to\bP(A_i)$ a.e.\ as $n\to\infty$
due to the ergodicity of $T$, it follows from \eqref{F-4.4} that
$\Gamma(\psi)=\beta\bigl(\sum_{i=1}^K\bP(A_i)\delta_{x_i}\bigr)=\beta(\psi_*\bP)
=E^\beta(\psi)$. For general $\ffi\in L^p(\Omega;M)$ choose simple functions $\psi_l$
such that $\bd_p(\ffi,\psi_l)\to0$ as $l\to\infty$. Since
$\bd_p(\Gamma(\ffi),\Gamma(\psi_l))\to0$ by (iii) and
$d(E^\beta(\ffi),E^\beta(\psi_l))\to0$ by Proposition \ref{P-2.2},
$\Gamma(\ffi)=E^\beta(\ffi)$ follows.
\end{proof}

When  $(\Omega,\cA)$ is a standard Borel space and $T$ is not
necessarily ergodic, the limit $\Gamma(\ffi)$ in Theorem \ref{T-4.2}
can be specified in terms of the $\beta$-conditional expectation of
$\ffi$ as follows.

\begin{thm}\label{T-4.3}
In Theorem $3.2$ assume that $(\Omega,\cA)$ is a standard Borel
space. Let $\cI:=\{A\in\cA:T^{-1}A=A\}$, the sub-$\sigma$-algebra
consisting of $T$-invariant sets. Then for every $\ffi\in
L^p(\Omega;M)$, $\Gamma(\ffi)$ is the $\beta$-conditional
expectation $E_\cI^\beta(\ffi)$ of $\ffi$ with respect to $\cI$
$($see Definition $\ref{D-2.8})$.
\end{thm}

\begin{proof}
Let $(\bP_\omega)_{\omega\in\Omega}$ be a disintegration of $\bP$ with respect to $\cI$, as
stated in Lemma \ref{L-2.5}. First, let $\psi$ be a simple function as
$\psi=\sum_{i=1}^K\1_{A_i}\delta_{x_i}$ with a measurable partition $\{A_1,\dots,A_K\}$ of
$\Omega$. From \eqref{F-4.4} we write
\begin{align}\label{F-4.7}
\beta(\mu_n^\psi(\omega))=\beta\Biggl(\sum_{i=1}^K\Biggl({1\over n}\sum_{k=0}^{n-1}
\1_{A_i}(T^k\omega)\Biggr)\delta_{x_i}\Biggr),\qquad n\in\bN.
\end{align}
Moreover, note that $\bP_\omega(A_j)$ is the usual conditional expectation of $\1_{A_j}$ with
respect to $\cI$ (see Lemma \ref{L-2.5}). Hence the usual individual ergodic theorem gives
$$
\lim_{n\to\infty}{1\over n}\sum_{k=0}^{n-1}\1_{A_i}(T^k\omega)
=\bP_\omega(A_i)\ \ \mbox{a.e.}
$$
Moreover, as in \eqref{F-2.3} we write
\begin{align}\label{F-4.8}
E_\cI^\beta(\psi)(\omega)=\beta\Biggl(\sum_{i=1}^K\bP_\omega(A_i)\delta_{x_i}\Biggr).
\end{align}
Therefore, from \eqref{F-4.7} and \eqref{F-4.8} together with \eqref{F-2.1} we have
\begin{align*}
d(\beta(\mu_n^\psi(\omega),E_\cI^\beta(\psi)(\omega))
&\le d_p^W\Biggl(\sum_{i=1}^K\Bigg({1\over n}\sum_{k=0}^{n-1}\1_{A_i}(T^k\omega)
\Biggr)\delta_{x_i},\sum_{i=1}^K\bP_\omega(A_i)\delta_{x_i}\Biggr) \\
&\le\Delta\Biggl[\sum_{i=1}^K\bigg|{1\over n}\sum_{k=0}^{n-1}
\1_{A_i}(T^k\omega)-\bP_\omega(A_i)\bigg|\Biggr]^{1/p}\longrightarrow\ 0\ \ \mbox{a.e.},
\end{align*}
where we have used Lemma \ref{L-1.3}. This implies that $\Gamma(\psi)=E_\cI^\beta(\psi)$.

For general $\ffi\in L^p(\Omega;M)$ choose simple functions $\psi_l$ such that
$\bd_p(\ffi,\psi_l)\to0$ as $l\to\infty$. By the above case,
$\Gamma(\psi_l)=E_\cI^\beta(\psi_l)$ for all $l$. Since Theorems \ref{T-4.2}\,(iii) and
\ref{T-2.9}\,(2) give $\bd_p(\Gamma(\ffi),\Gamma(\psi_l))\to0$ and
$\bd_p(E_\cI^\beta(\ffi),E_\cI^\beta(\psi_l))\to0$, we obtain $\Gamma(\ffi)=E_\cI^\beta(\ffi)$.
\end{proof}

\begin{thm}\label{T:order}
Assume that $M$ is equipped with a closed partial order and $\beta$
is monotone. Then $\Gamma$ is monotone, that is, for $\ffi,\psi\in
L^p(\Omega;M)$, $\ffi\le\psi$ $($i.e.,
$\ffi(\omega)\le\psi(\omega)$ a.e.$)$ implies
$\Gamma(\ffi)\le\Gamma(\psi)$.
\end{thm}
\begin{proof} Let $\ffi,\psi\in L^p(\Omega;M)$ and assume that
$\ffi(\omega)\leq \psi(\omega)$ a.e. Then
$\mu_n^\ffi(\omega)\leq \mu_n^\psi(\omega)$ a.e.\ and hence
$\beta(\mu_n^\ffi(\omega))\leq \beta(\mu_n^\psi(\omega))$ a.e.\ for
all $n$. By closedness of the partial order, letting $n\to\infty$
gives $\Gamma(\ffi)(\omega)\le\Gamma(\psi)(\omega)$ a.e. (When
$(\Omega,\cA)$ is a standard Borel space, the result also follows
from Theorems \ref{T-2.9}\,(5) and \ref{T-4.3}.)
\end{proof}

\begin{example}\label{E-4.5}\rm
(a)\enspace Consider the space $\bM_N$ of all $N\times N$ complex matrices with any norm
$|||\cdot|||$ (typically, the Hilbert-Schmidt norm), so $(\bM_N,d)$ with $d(X,Y):=|||X-Y|||$
is a Banach space. One can define the arithmetic mean map
$\cA:\cP^1(\bM_N,|||\cdot|||)\to\bM_N$ by
\begin{equation}\label{F-7}
\cA(\mu):=\int_{\bM_N}X\,d\mu(X),\qquad X\in\bM_N.
\end{equation}
Let $p\in[1,\infty)$ and $\mu,\nu\in\cP^p(\bM_N)$. For every $\pi\in\Pi(\mu,\nu)$ we have
\begin{align*}
\biggl[\int_{\bM_N\times\bM_N}|||X-Y|||^p\,d\pi(X,Y)\biggr]^{1/p}
&\ge\int_{\bM_N\times\bM_N}|||X-Y|||\,d\pi(X,Y) \\
&\ge\bigg|\bigg|\bigg|\int_{\bM_N\times\bM_N}(X-Y)\,d\pi(X,Y)\bigg|\bigg|\bigg| \\
&=|||\cA(\mu)-\cA(\nu)|||,
\end{align*}
which implies that $|||\cA(\mu)-\cA(\nu)|||\le d_p^W(\mu,\nu)$. Thus, $\cA$ restricted on
$\cP^p(\bM_N,|||\cdot|||)$ is a $p$-contractive barycentric map. For every
$\ffi\in L^p(\Omega;\bM_N)$ where $1\le p<\infty$, since
$$
\cA(\mu_n^\ffi(\omega))={1\over n}\sum_{k=0}^{n-1}\ffi(T^k\omega),\qquad\omega\in\Omega,
$$
Theorem \ref{T-4.2} in this case is the classical individual and mean ergodic theorems for
$\bM_N$-valued functions, where $\Gamma(\ffi)=E_\cI(\ffi)$, the usual conditional expectation
of $\ffi$ with respect to $\cI$.

(b)\enspace Let $\mathbb{P}=\mathbb{P}(\cH)$ be the
set of positive invertible operators on a Hilbert space $\cH$. A
natural metric on $\bbP$ is the Thompson metric $d_\T(A,B):=\|\log
A^{-1/2}BA^{-1/2}\|$, where $\|\cdot\|$ is the operator norm. Note
that $(\bbP,d_\T)$ is a complete metric space.  It turns out
\cite{LL14,LL16} that there exists a contractive barycentric map
$G:{\mathcal P}^1({\mathbb P})\to {\mathbb P}$, called the Karcher
barycenter, which is uniquely determined by
$$
X=G\left(\frac{1}{n}\sum_{j=1}^{n}\delta_{A_{j}}\right) \
\Longleftrightarrow\ \sum_{j=1}^n \log(X^{-1/2}A_j X^{-1/2})=0
$$
for all $n\in {\Bbb N}$ and $(A_1,\dots,A_n)\in
{\Bbb P}^n$. The L\"owner ordering $A\leq B$ is defined if $B-A$ is
a positive semidefinite operator on $\cH$, which is a closed partial
order on $\mathbb{P}$. It turns out \cite{Law} that the Karcher
barycenter is monotone for the L\"owner ordering.
  For $\ffi\in L^p(\Omega,\bbP)$ with $1\le
p<\infty$, note that
$$
G(\mu_n^\ffi(\omega))=G\Biggl({1\over
n}\sum_{k=0}^{n-1}\delta_{\ffi(T^k\omega)}\Biggr)
=G(\ffi(\omega),\ffi(T\omega),\dots,\ffi(T^{n-1}\omega)),\qquad\omega\in\Omega,
$$
which is the Karcher mean  of $\ffi(T^k\omega)$ ($0\le k\le n-1$).
Theorem \ref{T-4.2} says that
$$
\lim_{n\to\infty}G(\ffi,\ffi\circ T,\dots,\ffi\circ
T^{n-1})=\Gamma(\ffi) \ \,\mbox{a.e.\ and in metric $\bd_p$},
$$
where $\Gamma$ is a map from $L^p(\Omega;\bbP)$ onto $\{\ffi\in
L^p(\Omega,\bbP):\ffi\circ T=\ffi\}$. When $(\Omega,\cA)$ is a standard Borel space,
it follows from Theorem \ref{T-4.3} that $\Gamma(\ffi)=E_\cI^G(\ffi)$, the
$G$-conditional expectation of $\ffi$ with respect to $\cI$. By Theorem \ref{T:order}
the monotonicity of $\Gamma$ follows from that of $G$. Moreover, $\Gamma$ is monotone
and $\Gamma(\ffi^{-1})=\Gamma(\ffi)^{-1}$ as seen from $G(\mu^{-1})=G(\mu)^{-1}$ where
$\mu^{-1}$ is the push-forward of $\mu$ by $A\mapsto A^{-1}$ on $\bP$.
\end{example}

\section{Barycentric metric spaces and semiflows}
In this section, let $T$ be, as in Section 4, a $\bP$-preserving
measurable transformation on $(\Omega,\cA,\bP)$. Let $1\le p<\infty$
be fixed and denote by ${\mathcal C}_p(M)$ the set of all
$p$-contractive barycentric maps on the complete metric space $M$.
We note that  a metric space equipped with a contractive barycentric map is called
a \emph{barycentric metric space} and that there are many (distinct)
contractive barycentric maps on a metric space. For every $\beta\in
{\mathcal C}_p(M)$, let $\Gamma_\beta:L^p(\Omega;M)\to \{\ffi\in
L^p(\Omega;M):\ffi\circ T=\ffi\}$ be the map given in Theorem
\ref{T-4.2}. This naturally defines a two-variable map
\begin{eqnarray}\label{E:Cont}\Gamma:{\mathcal C}_p(M)\times L^p(\Omega;M)\to \{\ffi\in
L^p(\Omega;M):\ffi\circ T=\ffi\},\quad (\beta,\varphi)\mapsto
\Gamma_\beta(\varphi). \end{eqnarray} By Theorem \ref{T-4.2} (iii),
$\Gamma$ is continuous in variable $\varphi\in L^p(\Omega;M)$.  We
will construct a complete metric on ${\mathcal C}_p(M)$ such that
$\Gamma$ is continuous on ${\mathcal C}_p(M)\times L^p(\Omega;M)$
with respect to the product metric.

For ${\bf x}=(x_{1},\dots,x_{n})\in M^n$ and $\beta\in {\mathcal
C}_p(M)$, we write $\Delta({\bf x})$ for the
diameter of $\{x_1,\dots,x_n\}$, and $\beta({\bf
x}):=\beta\left(\frac{1}{n}\sum_{j=1}^{n}\delta_{x_{j}}\right)$.

\begin{prop}\label{L:MD}For   $\beta_{1},\beta_{2}\in {\mathcal C}_p(M)$, define
$$d_p(\beta_{1},\beta_{2})=\sup\limits_{\begin{subarray}{c}{\bf x}\in M^{n}, n\in {\Bbb N}\\
\Delta({\bf x})\neq 0\end{subarray}} \frac{d(\beta_{1}({\bf x}),
\beta_{2}({\bf x}))}{\Delta({\bf x})}.
$$
Then  $d_p(\beta_1,\beta_2)\le1$ for all
$\beta_1,\beta_2\in\cC_p(M)$ and $d_p$ is a complete metric on
${\mathcal C}_p(M)$.
\end{prop}

\begin{proof}
Let $z:=\beta_1({\bf x})$. We have
\begin{align*}
d(\beta_1({\bf x}),\beta_{2}({\bf x}))
&=d\Biggl(\beta_2(\delta_z),\beta_2\Biggl(\sum_{j=1}^n\delta_{x_j}\Biggr)\Biggr)
=d_p^W\Biggl(\delta_z,\sum_{j=1}^n\delta_{x_j}\Biggr) \\
&\le\Biggl[\frac{1}{n}\sum_{j=1}^{n}d^p(\beta_{1}({\bf x}),x_{j})\Biggr]^{1/p}
\leq\Biggl[\frac{1}{n^{2}}\sum_{j=1}^{n}\sum_{i=1}^{n}d^p(x_{i},x_{j})\Biggr]^{1/p}
\leq\Delta({\bf x}),
\end{align*}
and hence $d_p(\beta_{1},\beta_{2})$ is well defined with
$d_p(\beta_1,\beta_2)\le1$. It is straightforward to see that $d_p$
satisfies the metric properties. In particular,
$d_p(\beta_{1},\beta_{2})=0$ if and only if $\beta_1=\beta_2$, by
denseness of ${\mathcal P}_0(M)$ in ${\mathcal P}^p(M)$.

Assume that $\{\beta_k\}$ is a Cauchy sequence in ${\mathcal
C}_p(M)$. Let $\mu\in {\mathcal P}^p(M)$ and let  $\ve>0$. By
denseness of ${\mathcal P}_{0}(M)$ in ${\mathcal P}^p(M)$,  one can
find an $\bx=(x_1,\dots,x_n)\in M^n$ such that $d_p^W(\mu,\mu_0)<\ve$,
where $\mu_0:=\frac{1}{n}\sum_{j=1}^n\delta_{x_j}$. For every
$k,l\in\bN$ one has
\begin{align*}
d(\beta_k(\mu),\beta_l(\mu))
&\le d(\beta_k(\mu),\beta_k(\bx))+d(\beta_k(\bx),\beta_l(\bx))+d(\beta_l(\bx),\beta_l(\mu)) \\
&\le d(\beta_k(\bx),\beta_l(\bx))+ 2d_p^W(\mu, \mu_0)\\
 &\le d(\beta_k(\bx),\beta_l(\bx))+2\ve.
\end{align*}
There is a $k_0\in\bN$ such that $d(\beta_k(\bx),\beta_l(\bx))\le
d_p(\beta_k,\beta_l)\Delta({\bf x})\le\ve$ for all $k,l\ge k_0$.
Hence $d(\beta_k(\mu),\beta_l(\mu))\le3\ve$ for all $k,l\ge k_0$, so
$\{\beta_k(\mu)\}$ is Cauchy in $M$. Therefore, one can define
$\beta:\cP^p(M)\to M$ by
$$
\beta(\mu):=\lim_{k\to\infty}\beta_k(\mu)\in M.
$$
For every $x\in M$, since $\beta_k(\delta_x)=x$ for all $k$, we have
$\beta(\delta_x)=x$. For every $\mu,\nu\in\cP^p(M)$,
$$
d(\beta(\mu),\beta(\nu))=\lim_{k\to \infty}d(\beta_k(\mu),\beta_k(\nu))\le d_p^W(\mu,\nu).
$$
Hence $\beta\in\cC_p(M)$.

Next, we show that $d_p(\beta_k,\beta)\to0$. For every $\ve>0$
choose a $k_0\in\bN$ such that $d_p(\beta_k,\beta_l)\le\ve$ for all
$k,l\ge k_0$. For any $n\in\bN$ and $\bx\in M^n$ with
$\Delta(\bx)>0$,
$$
{d(\beta_k(\bx),\beta_l(\bx))\over\Delta(\bx)}\le
d_p(\beta_k,\beta_l)\le\ve, \qquad k,l\ge k_0.
$$
Since $d(\beta_l(\bx),\beta(\bx))\to0$ as $l\to\infty$, one has $
d(\beta_k(\bx),\beta(\bx))/\Delta(\bx)\le\ve$ for $k\ge k_0,$
which implies that $d_p(\beta_k,\beta)\le\ve$ for all $k\ge k_0$.
Hence $d_p(\beta_k,\beta)\to0$.
\end{proof}

\begin{thm}\label{T:gamma}
The map $\Gamma$ is continuous on $\cC_p(M)\times L^p(\Omega;M)$
with respect to the product metric of $d_p$ and $\bd_p$, that is,
for sequences $\{\beta_k\}$ in $\cC_p(M)$ and $\{\ffi_k\}$ in
$L^p(\Omega;M)$, if $\beta_k\to\beta\in\cC_p(M)$ in $d_p$ and
$\ffi_k\to\ffi\in L^p(\Omega;M)$ in $\bd_p$, then
$\Gamma_{\beta_k}(\ffi_k)\to\Gamma_\beta(\ffi)$ in $\bd_p$ as
$k\to\infty$. In particular, if $T$ is ergodic, then
$\lim_{k\to\infty}E^{\beta_k}(\ffi_k)=E^\beta(\ffi)$.

Furthermore, assume that $(\Omega,\cA)$ is a standard Borel space.
If $\beta_k,\beta\in\cC_p(M)$ and $\beta_k\to\beta$ in $d_p$, then
for every $\ffi\in L^p(\Omega;M)$,
$$\lim_{k\to\infty}\Gamma_{\beta_k}(\varphi)(\omega)
=\Gamma_\beta(\varphi)(\omega)\ \ a.e.
$$
\end{thm}

\begin{proof}
First, assume that $\varphi$ is a simple function with values
$x_{1},\dots,x_{K}$, and let $A_{j}:=\varphi^{-1}(x_{j})$ for $1\le
j\le K$. From the proof of Theorem \ref{T-4.2} (see the paragraph
containing \eqref{F-4.5}), we recall that for any
$\beta\in\cC_p(M)$,
\begin{align}\label{F-6.1}
\Gamma_\beta(\ffi)(\omega)
=\beta\Biggl(\sum_{i=1}^K\xi_i(\omega)\delta_{x_i}\Biggr) \ \
\mbox{a.e.}\ \ \omega\in\Omega,
\end{align}
where $\xi_i(\omega):=\lim_{n\to\infty}{1\over
n}\sum_{k=0}^{n-1}\1_{A_i}(T^k\omega)$. Assume that $\{\beta_k\}$ is
a sequence in ${\mathcal C}_p(M)$ converging to $\beta$. Since
$\xi_i(\omega)\ge0$ and $\sum_{i=1}^K\xi_i(\omega)=1$ a.e., one can
choose, for any $N\in\bN$ and for a.e.\ $\omega\in\Omega$,
$m_1(\omega),\dots,m_n(\omega)\in\bN\cup\{0\}$ such that
$m_i(\omega)$'s are measurable and
$$
\sum_{i=1}^nm_i(\omega)=N,\quad\bigg|\xi_i(\omega)-{m_i(\omega)\over
N}\bigg|\le{1\over N} \quad\mbox{a.e.}\ \ \omega\in\Omega.
$$
By Lemma \ref{L-1.3}, with $\Delta:=\diam\{x_1,\dots,x_K\}$,
\begin{align}
d_p^W\Biggl(\sum_{i=1}^K\xi_i(\omega)\delta_{x_i},{1\over
N}\sum_{i=1}^K m_i(\omega)\delta_{x_i}\bigg)
&\le\Delta\left({1\over2}\sum_{i=1}^K\bigg|\xi_i(\omega)-{m_i(\omega)\over
N}\bigg|\right)^{1/p} \nonumber\\
&\le\Delta\biggl({K\over2N}\biggr)^{1/p}\ \ \mbox{a.e.}
\label{F-6.2}
\end{align}
Then, from \eqref{F-6.1} and \eqref{F-6.2} it follows that
\begin{align*}
d(\Gamma_{\beta_k}(\varphi)(\omega),\Gamma_\beta(\varphi)(\omega))
&\le
d\Biggl(\beta_k\Biggl(\sum_{i=1}^K\xi_i(\omega)\delta_{x_i}\Biggr),
\beta_k\Biggl({1\over N}\sum_{i=1}^K m_i(\omega)\delta_{x_i}\Biggr)\Biggr) \\
&\quad+d\Biggl(\beta_k\Biggl({1\over N}\sum_{i=1}^K
m_i(\omega)\delta_{x_i}\Biggr),
\beta\Biggl({1\over N}\sum_{i=1}^K m_i(\omega)\delta_{x_i}\Biggr)\Biggr) \\
&\quad+d\Biggl(\beta\Biggl({1\over N}\sum_{i=1}^K
m_i(\omega)\delta_{x_i}\Biggr),
\beta\Biggl(\sum_{i=1}^K\xi_i(\omega)\delta_{x_i}\Biggr)\Biggr) \\
&\le2\Delta\biggl({K\over2N}\biggr)^{1/p}+\Delta d_p(\beta_k,\beta)
\ \ \mbox{a.e.}
\end{align*}
Therefore, letting $N\to\infty$ gives
$$
d(\Gamma_{\beta_k}(\varphi)(\omega),\Gamma_\beta(\varphi)(\omega))
\le\Delta d_p(\beta_k,\beta)\ \ \mbox{a.e.}
$$
Since $d_p(\beta_k,\beta)\to0$, we find that
$d(\Gamma_{\beta_k}(\varphi)(\omega),\Gamma_\beta(\varphi)(\omega))\to0$
a.e.\ and $\Gamma_{\beta_k}(\varphi)\to\Gamma_{\beta}(\varphi)$ in
$\bd_p$ as $k\to\infty$.

Now, let $\varphi\in L^{p}(\Omega, M)$ be arbitrary, and pick a
sequence of simple functions $\varphi_{m}:\Omega\to M$ such that
$\bd_p(\ffi_m,\ffi)\to0$. Since
$$
\bd_p(\Gamma_{\beta_{k}}(\ffi), \Gamma_{\beta}(\ffi)) \leq
\bd_p(\Gamma_{\beta_{k}}(\ffi),\Gamma_{\beta_{k}}(\ffi_{m})) +\bd_p(
\Gamma_{\beta_{k}}(\ffi_{m}), \Gamma_{\beta}(\ffi_{m}))
+\bd_p(\Gamma_{\beta}(\ffi_{m}),\Gamma_\beta(\ffi)),
$$
one has $\bd_p(\Gamma_{\beta_{k}}(\ffi), \Gamma_{\beta}(\ffi))\to0$
from the preceding paragraph and Theorem \ref{T-4.2}\,(iii). When
$\ffi_k,\ffi\in L^p(\Omega;M)$ and $\bd_p(\ffi_k,\ffi)\to0$, one has
\begin{align*}
\bd_p(\Gamma_{\beta_k}(\ffi_k),\Gamma_\beta(\ffi))
&\le\bd_p(\Gamma_{\beta_k}(\ffi_k),\Gamma_{\beta_k}(\ffi))
+\bd_p(\Gamma_{\beta_k}(\ffi),\Gamma_\beta(\ffi)) \\
&\le\bd_p(\ffi_k,\ffi)+\bd_p(\Gamma_{\beta_k}(\ffi),\Gamma_\beta(\ffi))
\,\longrightarrow\,0\quad\mbox{as $k\to\infty$}.
\end{align*}
If $T$ is ergodic, then by Theorem \ref{T-4.2},
$d(E^{\beta_k}(\ffi_k),E^\beta(\ffi))=\bd_p(\Gamma_{\beta_k}(\ffi_k),\Gamma_\beta(\ffi))\to0$.

Next, assume that $(\Omega,\cA)$ is a standard Borel space, and let
$(\bP_\omega)_{\omega\in\Omega}$ be a disintegration of $\bP$ with
respect to $\cI:=\{A\in\cA:T^{-1}A=A\}$. The following proof of the
a.e.\ convergence is based on the same method as that of Theorem \ref{T:martingale}.
Choose an $x_0\in M$ and let $\ffi_0:=\1_\Omega{x_0}$. Since
$x_0=\beta_k(\ffi_{0*}\bP_\omega)$ for all $k$, note that for every
$\ffi\in L^p(\Omega;M)$,
\begin{align*}
d(\Gamma_{\beta_k}(\ffi)(\omega),\Gamma_{\beta_l}(\ffi)(\omega))
&\le d(\Gamma_{\beta_k}(\ffi)(\omega),x_0)+d(\Gamma_{\beta_l}(\ffi)(\omega),x_0) \\
&=d(\beta_k(\ffi_*\bP_\omega),\beta_k(\ffi_{0*}\bP_\omega))
+d(\beta_l(\ffi_*\bP_\omega),\beta_l(\ffi_{0*}\bP_\omega)) \\
&\le2\biggl[\int_\Omega
d^p(\ffi(\tau),x_0)\,d\bP_\omega(\tau)\biggr]^{1/p} \ \ \mbox{a.e.}\
\omega,
\end{align*}
where we have used Theorem \ref{T-4.3} and Lemma \ref{L-2.7}\,(2). Since
Lemma \ref{L-2.5}\,(iii) gives
$$
\int_\Omega\biggl[\int_\Omega
d^p(\ffi(\tau),x_0)\,d\bP_\omega(\tau)\biggr]d\bP(\omega)=\bd_p^p(\ffi,\ffi_0)<\infty,
$$
one can define $W(\ffi)\in L^p(\Omega;\bR)$ by
$$
W(\ffi)(\omega):=\lim_{m\to\infty}\sup_{k,l\ge m}
d(\Gamma_{\beta_k}(\ffi)(\omega),\Gamma_{\beta_l}(\ffi)(\omega))
\ \ \mbox{a.e.}
$$
When $\ffi$ is a simple function, since
$\lim_{k\to\infty}\Gamma_{\beta_k}(\ffi)(\omega)$ exists a.e.\ as
shown in the first paragraph of the proof, we have $W(\ffi)=0$ as an
element of $L^p(\Omega;\bR)$. Now it suffices to prove that $W$ is
continuous from $L^p(\Omega;M)$ into $L^p(\Omega;\bR)$. Indeed, it
then follows that for every $\ffi\in L^p(\Omega;M)$, $W(\ffi)=0$ and
hence $\lim_{k\to\infty}\Gamma_{\beta_k}(\ffi)(\omega)$ exists a.e.

To prove the above stated continuity of $W$, note that for every
$\ffi,\psi\in L^p(\Omega;M)$ and every $k,l\ge1$,
\begin{align*}
d(\Gamma_{\beta_k}(\ffi)(\omega),\Gamma_{\beta_l}(\ffi)(\omega))
&\le
d(\Gamma_{\beta_k}(\ffi)(\omega),\Gamma_{\beta_k}(\psi)(\omega))
+d(\Gamma_{\beta_k}(\psi)(\omega),\Gamma_{\beta_l}(\psi)(\omega)) \\
&\qquad+d(\Gamma_{\beta_l}(\psi)(\omega),\Gamma_{\beta_l}(\ffi)(\omega)),
\end{align*}
from which we find that
\begin{align*}
|W(\ffi)(\omega)-W(\psi)(\omega)|
&\le2\sup_{k\ge1}d(\Gamma_{\beta_k}(\ffi)(\omega),\Gamma_{\beta_k}(\psi)(\omega)) \\
&=2\sup_{k\ge1}d(\beta_k(\ffi_*\bP_\omega),\beta_k(\psi_*\bP_\omega)) \\
&\le2\biggl[\int_\Omega
d^p(\ffi(\tau),\psi(\tau))\,d\bP_\omega(\tau)\biggr]^{1/p}
\end{align*}
due to Lemma \ref{L-2.7}\,(2). Therefore, by Lemma
\ref{L-2.5}\,(iii), $W(\ffi)-W(\psi)$ is in $L^p(\Omega;\bR)$ and
$$
\|W(\ffi)-W(\psi)\|_p\le2\bd_p(\ffi,\psi),
$$
implying the desired continuity of $W$.
\end{proof}

\begin{remark}\label{R-6.3}\rm
Assume that $(\Omega,\cA)$ is a standard Borel space and $\cB$ is a
sub-$\sigma$-algebra of $\cA$. In view of Theorem \ref{T-2.9} we
have two-variable map
$$
E_\cB:\cC_p(M)\times L^p(\Omega;M)\to L^p(\Omega,\cB,\bP;M),\qquad
(\beta,\ffi)\mapsto E_\cB^\beta(\ffi).
$$
In the same way as in the proof of Theorem \ref{T:gamma}, one can
see that $E_\cB$ is continuous on $\cC_p(M)\times L^p(\Omega;M)$
with respect to the product metric and that if $\beta_k\to\beta$ in
$\cC_p(M)$ then $E_\cB^{\beta_k}(\ffi)(\omega)\to
E_\cB^\beta(\ffi)(\omega)$ a.e.\ for every $\ffi\in L^p(\Omega;M)$.
When $\cB=\cI$, this is the latter assertion of Theorem \ref{T:gamma}.
\end{remark}

In the remaining of this section we assume that $(M,d)$ is a global NPC space.
For any $x,y\in M$, there exists a unique minimal geodesic
$\gamma_{x,y}:[0,1]\to M$ such that
$\gamma_{x,y}(0)=x$ and $\gamma_{x,y}(1)=y$. Denote
$x\#_ty:=\gamma_{x,y}(t)$, $t\in [0,1]$. We note that
$x\#y:=x\#_{1/2}y$ is the unique midpoint between $x$ and $y$. One
can see that
\begin{eqnarray}\label{E:cond}(x\#_s y)\#_r(x\#_t y)=x\#_{(1-r)s+rt}
y,\qquad x=x\#_ty\,\iff\,x=y.
\end{eqnarray}
Every global NPC space satisfies the following uniform convexity
(cf.\ \cite{Ku}): for $x,y,z\in M$, $t\in [0,1]$ and $q\geq 2$,
\begin{eqnarray}\label{uniform}
d^q(z,x\#_ty)\leq
(1-t)d^q(z,x)+td^q(z,y)-\frac{k_q}{2}\,t(1-t)d^{q}(x,y),
\end{eqnarray}
where $k_2=2$ and for $q>2$,
$
k_q=\frac{8}{2^q}\frac{1+\tau_q^{q-1}}{(1+\tau_q)^{q-1}}
$
and $\tau_q\in (1,\infty)$ is the unique solution to $x^{q-1}+(1-q)x+2-q=0$.

\begin{definition}
Let $1\leq p<\infty$. For $x\in M$, $t\in [0,1]$ and
$\mu \in {\mathcal P}^{p}(M)$, define $x\#_{t} \mu\in {\mathcal
P}^{p}(M)$ by $x\#_{t} \mu := f_{*}(\mu)$, where $f:M\to M$ is the
contraction mapping $f(a)=x\#_{t}a$.
\end{definition}

Note that $x\#_{0} \mu=\delta_{x}$ and $x\#_{1} \mu=\mu$, and
$x\#_{t}\mu=\frac{1}{n}\sum_{j=1}^{n}\delta_{x\#_{t}a_{j}}$ for $
\mu=\frac{1}{n}\sum_{j=1}^{n}\delta_{a_{j}}$. One can directly see
that $x\#_t(x\#_s\mu)=x\#_{st}\mu$ for $s,t\in [0,1]$ and $\mu\in
{\mathcal P}^{p}(M)$, which implies that for any fixed $x\in M$,
$(t,\mu)\mapsto x\#_t\mu$ is a semiflow on ${\mathcal P}^{p}(M)$
under the multiplicative semigroup on $[0,1]$. We note that $1$ is
the identity on the semigroup. The map $t\mapsto e^{-t}$ is a
homeomorphic isomorphism from the additive semigroup ${\Bbb
R}_+:=[0,\infty)$ onto $(0,1]$, so $(t,\mu)\mapsto x\#_{e^{-t}}\mu$
($t\ge0$) becomes an additive semiflow.

Recall the metric space ${\mathcal C}_{p}(M)$ of $p$-contractive
barycentric maps in Proposition \ref{L:MD}.

\begin{thm}\label{T:main} Let $1\leq p<\infty$.
There exists a continuous semiflow $\Phi_p:(0,1]\times
 {\mathcal C}_p(M)\to {\mathcal C}_p(M)$ satisfying
\begin{eqnarray}\label{E:char}
x=\Phi_p(t,\beta)(\mu)\ \iff\ x=\beta(x\#_t\mu)
\end{eqnarray}
for $t\in (0,1]$, $\beta\in {\mathcal C}_p(N)$ and
$\mu\in {\mathcal P}^{p}(M)$. Furthermore, for every $\beta\in
{\mathcal C}_p(M)$,
\begin{eqnarray}\label{p-uniform}
d_p\left(\Phi_p(t,\beta), \Phi_p(s,\beta)\right)\leq
\left[\frac{k_{2p}(s+t)+2(2-k_{2p})}{4}\right]^{\frac{1}{2p}}.
\end{eqnarray}
In particular, $\lim_{t\to 0^+}\Phi_1(t,\beta)=\lambda$
for every $\beta\in {\mathcal C}_{1}(M)$, where $\lambda:{\mathcal
P}^1(M)\to M$ is the canonical barycentric map on the global NPC space
$M$ given in $(\ref{E:least})$.
\end{thm}

\begin{proof}
Let $t\in (0,1]$ and $\beta\in {\mathcal C}_p(M)$. Let $\mu\in {\mathcal P}^{p}(M)$.
Define $F:M\to M$ by $F(x):= \beta(x\#_{t}\mu)$.  We shall show that $F$ is
a strict contraction on $M$. If $\mu=\frac{1}{n}\sum_{j=1}^{n}\delta_{a_{j}}$,
\begin{align*}
d(F(x),F(y))&=d\left(\beta\left(\frac{1}{n}\sum_{j=1}^{n}\delta_{x\#_{t}a_{j}}\right),
\beta\left(\frac{1}{n}\sum_{j=1}^{n}\delta_{y\#_{t}a_{j}}\right)\right)\\
&\leq \left[\frac{1}{n}\sum_{j=1}^{n}d^p(x\#_{t}a_{j},
y\#_{t}a_{j})\right]^{1/p}\leq (1-t)d(x,y),
\end{align*}
where the last inequality follows from $d(z\#_ty,x\#_ty)\leq
(1-t)d(z,x)$.  For general $\mu\in {\mathcal P}^{p}(M)$,
 pick a sequence
$\{\mu_{n}\}\subset {\mathcal P}_{0}(M)$ converging to $\mu$ in
${\mathcal P}^{p}(M)$. Then
$$
d(\beta(x\#_t\mu_n),
\beta(x\#_t\mu))\leq d_p^W(x\#_t{\mu_n},x\#_t\mu)\leq
d_p^W(\mu_n,\mu)\to 0\quad\mbox{as $n\to\infty$},
$$
and hence
$d(F(x),F(y))=\lim_{n\to\infty}d(\beta(x\#_{t}\mu_{n}),\beta(y\#_{t}\mu_{n}))\leq(1-t)\lim_{n\to\infty}d(x,y)=(1-t)d(x,y)$
which shows that $F$ is a strict contraction and hence $x=F(x)$ has
a unique solution. Let $\Phi_p(t,\beta)(\mu)$ denote the unique
fixed point of $F$.

We will show that $\Phi_p(t,\beta)\in {\mathcal C}_p(M)$ for all
$t\in (0,1]$ and $\beta\in {\mathcal C}_p(M)$. By definition,
$\Phi_p(1,\beta)=\beta$ for all $\beta\in {\mathcal C}_p(M)$. Fix
$t\in (0,1)$. Let $z\in M$ and put $x=\Phi_p(t,\beta)(\delta_z)$.
Then $x=\beta(x\#_t\delta_z)=\beta(\delta_{x\#_tz})=x\#_tz$ and
hence $\Phi_p(t,\beta)(\delta_z)=x=z$ by {\eqref{E:cond}}. Let
$x=\Phi_p(t,\beta)(\mu)$ and $y=\Phi_p(t,\beta)(\nu)$. Then
\begin{align*}
d(x,y)&=d(\beta(x\#_t\mu),\beta(y\#_t\nu))\leq d_p^W(x\#_t\mu,y\#_t\nu)\\
&\leq \bigl[(1-t)d^p(x,y)+td_p^W(\mu,\nu)^p\bigr]^{1/p},
\end{align*}
which implies that $d(x,y)\leq d_p^W(\mu,\nu)$. Therefore,
$\Phi_p(t,\beta):{\mathcal P}^{p}(M)\to M$ is a $p$-contractive
barycentric map.

We will prove that $\Phi_p$ is a continuous semiflow on ${\mathcal
C}_{p}(M)$. Let $x=\Phi_p(t,\Phi_p(s,\beta))(\mu)$. Then
$x=\Phi_p(s,\beta)(x\#_t\mu)$, which is equivalent  to
$x=\beta(x\#_s(x\#_t\mu))=\beta(x\#_{st}\mu)$, that is,
$x=\Phi_p(st,\beta)(\mu)$. To see the continuity of $\Phi_p$, let ${\bf
a}=(a_{1},\dots,a_{n})\in M^n$, and set $\mu_{\bf
a}:=\frac{1}{n}\sum_{j=1}^{n}\delta_{a_j}$ for
notational simplicity. Let $x=\Phi_p(t,\beta_1)({\mu_{\bf a}})$
and $y=\Phi_p(s,\beta_2)(\mu_{\bf a})$. Then by the triangle inequality
and \cite[Proposition 3.8\,(2)]{LLL},
\begin{align*}
d(x,y)&=d(\beta_1(x\#_{t}\mu_{\bf a}),\beta_2(y\#_{s}\mu_{\bf a})) \\
&\leq d(\beta_1(x\#_{t}\mu_{\bf a}),\beta_2(x\#_{t}\mu_{\bf a}))
+d(\beta_2(x\#_{t}\mu_{\bf a}),\beta_2(y\#_{s}\mu_{\bf a})) \\
&\leq d_p(\beta_1,\beta_2)\Delta(x\#_{t}a_{1},\dots,x\#_{t}a_{n})
+\frac{1}{n}\sum_{j=1}^{n}d(x\#_{t}a_{j},y\#_{s}a_{j})\\
&\leq td_p(\beta_1,\beta_2)\Delta({\bf a})
+\frac{1}{n}\sum_{j=1}^{n}\left[(1-t)d(x,y)+|t-s|d(y,a_{j})\right]\\
&=td_p(\beta_1,\beta_2)\Delta({\bf a})+(1-t)d(x,y)+\frac{|t-s|}{n}\sum_{j=1}^{n}d(y,a_{j}),
\end{align*}
where $\Delta(x\#_{t}a_{1},\dots,x\#_{t}a_{n})\leq t\Delta({\bf a})$
follows from $d(x\#_{t}a_{i}, x\#_{t}a_{j})\leq td(a_{i},a_{j})$.
Moreover, since $ d(y,a_{j})=d(\Phi_p(s,\beta_2)(\mu_{\bf a}),
\Phi_p(s,\beta_2)(\delta_{a_{j}})) \leq d_p^W(\mu_{\bf
a},\delta_{a_{j}}) \leq \Delta({\bf a}), $ we find that $ d(x,y)\leq
d_p(\beta_1,\beta_2) \Delta({\bf a})+\frac{|t-s|}{t}\Delta({\bf a}).
$ This implies that
\begin{align*}
d(\Phi_p(t,\beta_1),\Phi_p(s,\beta_2))
&= \sup\limits_{\begin{subarray}{c}{\bf a}\in M^{n}, n\in {\Bbb N}\\
\Delta({\bf a})\neq 0\end{subarray}}
\frac{d(\Phi_p (t,\beta_1)(\mu_{\bf a}),
\Phi_p(s,\beta_2)(\mu_{\bf a}))}{\Delta({\bf a})}\\
&\leq d_p(\beta_1,\beta_2)+\frac{|t-s|}{t},
\end{align*}
which shows  continuity of $\Phi_p$.

Next, we shall show (\ref{p-uniform}). Let
$x=\Phi_p(t,\beta)(\mu_{\bf a})$. For any $z\in M$, from (\ref{uniform}) we have
\begin{align*}
d^{2p}(z,x)&=d^{2p}\Biggl(\beta(\delta_z),\beta\Biggl({1\over
n}\sum_{j=1}^n\delta_{x\#_ta_j}\Biggr)\Biggr)
\leq d_{p}^W\Biggl(\delta_z,\Biggl({1\over n}\sum_{j=1}^n\delta_{x\#_ta_j}\Biggr)\Biggr)^{2p}  \\
&\leq \frac{1}{n}\sum_{i=1}^{n}d^{2p}(z,x\#_{t}a_{i})\\
&\leq\frac{1}{n}\sum_{i=1}^{n}\biggl[(1-t)d^{2p}(z,x)+td^{2p}(z,a_{i})
-\frac{k_{2p}}{2}(1-t)td^{2p}(x,a_{i})\biggr]\\
&=(1-t)d^{2p}(z,x)+\frac{t}{n}\sum_{i=1}^{n}d^{2p}(z,a_{i})
-\frac{k_{2p}(1-t)t}{2n}\sum_{i=1}^{n}d^{2p}(x,a_{i}),
\end{align*}
and hence
\begin{eqnarray}\label{E:npc}
d^{2p}(z,x)\leq
\frac{1}{n}\sum_{i=1}^{n}d^{2p}(z,a_{i})-\frac{k_{2p}(1-t)}{2n}\sum_{i=1}^{n}d^{2p}(x,a_{i}).
\end{eqnarray}
Applying this with $y=\Phi_p(s,\beta)(\mu_{\bf a})$ leads to
\begin{align*}
d^{2p}(y,x)&\leq \frac{1}{n}\sum_{i=1}^{n}d^{2p}(y,a_{i})-\frac{k_{2p}(1-t)}{2n}
\sum_{i=1}^{n}d^{2p}(x,a_{i}), \\
d^{2p}(x,y)&\leq \frac{1}{n}\sum_{i=1}^{n}d^{2p}(x,a_{i})-\frac{k_{2p}(1-s)}{2n}
\sum_{i=1}^{n}d^{2p}(y,a_{i}).
\end{align*}
Summing these yields
\begin{align*}
2d^{2p}(x,y)&\leq\frac{2-k_{2p}(1-s)}{2n}\sum_{j=1}^nd^{2p}(y,a_j)
+\frac{2-k_{2p}(1-t)}{2n}\sum_{j=1}^nd^{2p}(x,a_j)\\
&\leq\frac{k_{2p}(s+t)+2(2-k_{2p})}{2}\,\Delta({\bf a})^{2p}.
\end{align*}

Finally, assume that $p=1$. Since $k_2=2$, it follows from (\ref{p-uniform}) that
$\{{\Phi_1}(t,\beta)\}_{t\in (0,1]}$ is a Cauchy net in the complete metric space
${\mathcal C}_{1}(M)$. Let
$$
\beta_{0}:=\lim_{t\to 0^+}\Phi_1(t,\beta).
$$
By (\ref{E:npc}) we have for every $z\in M$
$$
d^{2}(z,\Phi_1(t,\beta)(\mu_{\bf
a}))\leq
\frac{1}{n}\sum_{i=1}^{n}d^{2}(z,a_{i})
-\frac{(1-t)}{n}\sum_{i=1}^{n}d^{2}(\Phi_1(t,\beta)(\mu_{\bf a}),a_{i}).
$$
By taking the limit of both sides of the above as $t\to 0^{+}$, we
have
$$
d^{2}(z,\beta_0(\mu_{\bf a}))\leq \frac{1}{n}\sum_{i=1}^{n}d^{2}(z,a_{i})
-\frac{1}{n}\sum_{i=1}^{n}d^{2}(\beta_0(\mu_{\bf
a}),a_{i}),
$$
which implies that $\frac{1}{n}\sum_{i=1}^{n}d^{2}(\beta_0(\mu_{\bf
a}),a_{i})\leq \frac{1}{n}\sum_{i=1}^{n}d^{2}(z,a_{i})$ for any
$z\in M$. This shows that $\beta_0(\mu_{\bf a})= \lambda(\mu_{\bf
a})$ for all ${\bf a}=(a_{1},\dots,a_n)\in M^n$ and $n\in {\Bbb N}$.
By continuity and denseness of ${\mathcal P}_{0}(M)$ in ${\mathcal
P}^1(M)$, we have $\beta_0(\mu)=\lambda(\mu)$ for all $\mu\in
{\mathcal P}^1(M)$.
\end{proof}

\begin{remark}\rm
The canonical barycenter $\lambda:{\mathcal P}^1(M)\to M$  is then
the global attractor of the  semiflow $\Phi_p$ on ${\mathcal
C}_{p}(M)$ for $p=1$.
\end{remark}

By Theorems \ref{T:gamma} and \ref{T:main},

\begin{cor}
Let $\beta\in {\mathcal C}_1(M)$ and set $\beta_t:=\Phi_1(t,\beta)$.
Then for every $\ffi\in L^1(\Omega;M)$,
$$
\bd_1\bigl(\Gamma_{\beta_t}(\varphi),\Gamma_\lambda(\ffi)\bigr)\,\longrightarrow\,0
\quad\mbox{and}\quad
d\bigl(\Gamma_{\beta_t}(\varphi)(\omega),\Gamma_\lambda(\varphi)(\omega)\bigr)
\,\longrightarrow\,0\ \ a.e.
$$
as $t\to0^+$. If $T$ is ergodic, then $\lim_{t\to
0^+}E^{\beta_t}(\ffi)=E^\lambda(\ffi)$.
\end{cor}

\begin{remark}\rm
(1)\enspace It does not seem easy to show the convergence of the net
$\{{\Phi_p}(t,\beta)\}_{t\in (0,1]}$ as $t\to0^+$ for general $p>1$.
Although the minimizer
$$
\lambda_p(\mu_{\bf a}):=\underset{x\in
M}{\argmin}\sum_{i=1}^nd^{2p}(x,a_i)
$$
exists uniquely for every ${\bf a}=(a_{1},\dots,a_n)\in M^n$ and
$n\in {\Bbb N}$ from the uniform convexity in (\ref{uniform}) and
from \cite[Proposition 1.7]{St}, to the best of our knowledge, its
$p$-contractive property is unknown.

(2)\enspace For $\beta_1,\beta_2\in {\mathcal C}_p(M)$ and $t\in
[0,1],$ define
$(\beta_{1}\#_t\beta_2)(\mu):=\beta_{1}(\mu)\#_t\beta_{2}(\mu)$ for
$\mu\in {\mathcal P}^p(M).$ Then it is direct to see that the map
$t\mapsto \beta_{1}\#_t\beta_2$ is a minimal geodesic in ${\mathcal
C}_{p}(M)$ with respect to the complete metric $d_p,$ and also that
$(\beta_{1}\#_s\beta_2)\#_{r}(\beta_{1}\#_t\beta_2)=\beta_{1}\#_{(1-r)s+rt}\beta_2$
and  $d_p(\beta_{1}\#_t\beta_2,\beta_3\#_t\beta_4)\leq
(1-t)d_{p}(\beta_{1},\beta_3)+td_{p}(\beta_{2},\beta_4).$
 This shows
that  $({\mathcal C}_p(M), d_p)$ is a convex metric space
\cite{LL2,LLL}, or a Busemann space without  the  uniqueness of
geodesics or midpoints. Navas' approach in \cite{Na} allows us to
define on $\cC_p(M)$ the contractive barycentric map of Es-Sahib
and  Heinich \cite{EH}.
\end{remark}

\section{Large deviation principle}

First, recall the general formulation of the large deviation
principle (LDP) (cf.\ \cite{DZ}).  Let $\cX$ be a metric space and
$\cB(\cX)$ the Borel $\sigma$-algebra on $\cX$. Let
$(\mu_n)_{n=1}^\infty$ be a sequence of Borel probability measures
on $\cX$. A function $I:\cX\to[0,\infty]$ is called a {\it rate
function} if $I$ is lower semicontinuous, that is, for every
$\alpha\in[0,\infty)$ the level set $\{x\in\cX:I(x)\le\alpha\}$ is
closed. A {\it good rate function} is a rate function
$I:\cX\to[0,\infty]$ whose level sets are compact for all
$\alpha\in[0,\infty)$. It is said that $(\mu_n)$ satisfies the {\it
LDP} (in the scale $1/n$) with a rate function $I$, if for every
$\Gamma\in\cB(\cX)$,
$$
-\inf_{x\in\Gamma^\circ}I(x)
\le\liminf_{n\to\infty}{1\over n}\log\mu_n(\Gamma)
\le\limsup_{n\to\infty}{1\over n}\log\mu_n(\Gamma)
\le-\inf_{x\in\overline\Gamma}I(x),
$$
where $\Gamma^\circ$ and $\overline\Gamma$ denote the interior and the closure of $\Gamma$,
respectively.

Let $(\Omega,\cA,\bP)$ be a probability space. Let $\Sigma$ be a
Polish space and $\cP(\Sigma)$ be the set of Borel probability
measures on $\Sigma$ equipped with the weak topology. Note that the weak
topology on $\cP(\Sigma)$ is metrizable with the L\'evy-Prokhorov
metric $\rho$ and $(\cP(\Sigma),\rho)$ becomes a Polish space. Let
${\mathbf X}=(X_1,X_2,\dots)$ be a sequence of i.i.d.\
$\Sigma$-valued random variables and $\mu_0\in\cP(\Sigma)$ be their
equal distribution, i.e., $\mu_0(B)={\bf P}(X_i^{-1}(B))$ for all
$i\in {\Bbb N}$ and $B\in\cB(\Sigma)$. We define the {\it
empirical measure}
$$
\mu_n^{\mathbf X}(\omega):={1\over n}\sum_{i=1}^n\delta_{X_i(\omega)},\qquad n\in\bN,
$$
and consider the distribution $\widehat\mu_n$ of
$\mu_n^{\mathbf X}:\Omega\to \cP(\Sigma)$, i.e.,
for Borel sets $\Gamma\subset\cP(\Sigma)$,
$$
\widehat\mu_n(\Gamma):=\bP(\mu_n^{\mathbf X}\in\Gamma)
=\mu_0^{\times n}\biggl(\biggl\{(x_1,\dots,x_n)\in\Sigma^n:
{1\over n}\sum_{i=1}^n\delta_{x_i}\in\Gamma\biggr\}\biggr).
$$
Then the celebrated {\it Sanov theorem} is

\begin{thm}\label{T-5.1}
The distributions $(\widehat\mu_n)$ of the empirical measures
$(\mu_n^{\mathbf X})$ satisfies the LDP with the relative entropy
functional $S(\cdot\|\mu_0)$ as the good rate function, where the
relative entropy (or the Kullback-Leibler divergence)
$S(\mu\|\mu_0)$ is defined by
$$
S(\mu\|\mu_0):=\begin{cases}
\int_\Sigma\log{d\mu\over d\mu_0}\,d\mu & \mbox{if $\mu\ll\mu_0$ $($absolutely continuous$)$}, \\
\infty & \mbox{otherwise}.
\end{cases}
$$
\end{thm}

Now, assume that $(M,d)$ be a complete metric space and let ${\mathbf
X}=(X_1,X_2,\dots)$ be a sequence of i.i.d.\ $M$-valued random
variables. Assume that the distribution $\mu_0$ of $X_i$ is in
$\cP^\infty(M)$, i.e., $X_i\in L^\infty(\Omega;M)$. Since a strongly
measurable $M$-valued function has a separable range except on a
$\bP$-null set, one can choose a separable closed subset $M_0$ of
$M$ such that $X_i(\omega)\in M_0$ for all $i\in\bN$ and a.e.\
$\omega\in\Omega$ (or $\mu_0$ is supported on $M_0$). Moreover,
choose an $x_0\in M$ and let
$\alpha:={\mathrm{ess\,sup}}_{\omega\in\Omega}d(X_1(\omega),x_0)<\infty$.
Then $X_i(\omega)\in\Sigma:=\{x\in M_0:d(x,x_0)\le\alpha\}$ for all
$i\in\bN$ and a.e.\ $\omega\in\Omega$. Note that $\Sigma$ ($\subset
M$) is a Polish space, and we may assume that $X_i$'s are
$\Sigma$-valued random variables. Hence the Sanov LDP holds for the
sequence ${\mathbf X}=(X_1,X_2,\dots)$.

Let $1\le p<\infty$ and $\beta:\cP^p(M)\to M$ be a $p$-contractive
barycentric map. Note that $\cP(\Sigma)$ is a subset of $\cP^p(M)$.
Since $\Sigma$ is bounded, it follows from \cite[Theorem 7.12]{Vi}
that the $d_p^W$-topology on $\cP(\Sigma)$ coincides with the weak
topology on $\cP(\Sigma)$. Hence
$\mu\in\cP(\Sigma)\mapsto\beta(\mu)$ is a continuous map from
$\cP(\Sigma)$ equipped with the weak topology to $(M,d)$.
 Note that the
push-forward of $\widehat\mu_n$ by $\beta|_{\cP(\Sigma)}$ is the
distribution of $\beta(\mu_n^{\mathbf X})$, i.e., for every
$\Gamma\in\cB(M)$,
\begin{align*}
\widehat\mu_n(\{\mu\in\cP(\Sigma):\beta(\mu)\in\Gamma\})
&=\bP(\beta(\mu_n^{\mathbf X})\in\Gamma) \\
&=\bP\biggl(\biggl\{\omega\in\Omega: \beta\biggl({1\over
n}\sum_{i=1}^n\delta_{X_i(\omega)}\biggr)\in\Gamma\biggr\}\biggr).
\end{align*}
Therefore, from Theorem \ref{T-5.1}, applying the contraction principle for LDP (see
\cite[Theorem 4.2.1]{DZ}) with the continuous map $\beta:\cP(\Sigma)\to M$, we have the
following:

\begin{thm}\label{T-5.2}
With the above definitions and assumptions, the distribution of the $M$-valued random variable
$\beta(\mu_n^{\mathbf X})=\beta\bigl({1\over n}\sum_{i=1}^n\delta_{X_i}\bigr)$ satisfies the
LDP with the good rate function
\begin{align}\label{F-5.1}
I(x):=\inf\{S(\mu\|\mu_0):\mu\in\cP(\Sigma),\,x=\beta(\mu)\},\qquad x\in M.
\end{align}
That is, for every $\Gamma\in\cB(M)$,
\begin{align}
-\inf_{x\in\Gamma^\circ}I(x)
&\le\liminf_{n\to\infty}{1\over n}\log
\bP\bigl(\beta(\mu_n^{\mathbf X})\in\Gamma\bigr) \nonumber\\
&\le\limsup_{n\to\infty}{1\over n}\log
\bP\bigl(\beta(\mu_n^{\mathbf X})\in\Gamma\bigr) \label{F-5.2}
\le-\inf_{x\in\overline\Gamma}I(x).
\end{align}
\end{thm}

The above LDP is a stronger version of the strong law of large numbers for the $\beta$-value
$\beta\bigl({1\over n}\sum_{i=1}^n\delta_{X_i}\bigr)$ of the empirical measure, given in
\cite[Proposition 6.6]{St}. Let $x_0:=\beta(\mu_0)$. Since $S(\cdot\|\mu_0)$ is a good rate
function on $\cP(\Sigma)$, for every $x\in M$ with $I(x)<\infty$ there is a $\mu\in\cP(\Sigma)$
such that $x=\beta(\mu)$ and $I(x)=S(\mu\|\mu_0)$. Therefore, from the strict positivity of
the relative entropy, we see that $I(x)>0$ whenever $x\ne x_0$. For any $\eps>0$ take a closed
set $F:=\{x\in M:d(x,x_0)\ge\eps\}$; then the LDP upper bound in \eqref{F-5.2} gives
$$
\limsup_{n\to\infty}{1\over n}\log
\bP\bigl(\beta(\mu_n^{\mathbf X})\in F\bigr)
\le-\inf_{x\in F}I(x)<-\alpha\quad\mbox{for some $\alpha>0$},
$$
since $I$ is a good rate function. This implies that
$$
\sum_{n=1}^\infty\bP\bigl(\bigl\{\omega:d(\beta(\mu_n^{\mathbf X})(\omega),x_0)
\ge\eps\bigr\}\bigr)<\infty,
$$
so the Borel-Cantelli lemma yields that
$$
\bP\biggl(\limsup_{n\to\infty}\bigl\{\omega:d(\beta(\mu_n^{\mathbf
X})(\omega),x_0) \ge\eps\bigr\}\biggr)=0,
$$
which implies that $d(\beta(\mu_n^{\mathbf X})(\omega),x_0)\to0$ as $n\to\infty$.
We thus have the strong law of large numbers in \cite[Proposition 6.6]{St}.

\begin{cor}\label{C-5.3}
Let $X_1,X_2,\dots$ be a sequence of i.i.d.\ $M$-valued random variables having the
distribution $\mu_0\in\cP^\infty(M)$. Then
$$
\beta\Biggl({1\over n}\sum_{i=1}^n\delta_{X_i(\omega)}\Biggr)
\ \longrightarrow\ \beta(\mu_0)\ \ \mbox{a.e.}\quad\mbox{as $n\to\infty$}.
$$
\end{cor}

\begin{remark}\label{R-5.4}\rm
A point of the above argument is that although the Sanov LDP is
concerned with the weak topology on $\cP(M)$, the contractive
barycentric map $\beta$ on $\cP^p(M)$ is continuous with respect to
the Wasserstein distance $d_p^W$ so that $\beta$ is not necessarily
continuous with respect to the weak topology. This is the reason why
we have to assume that the i.i.d.\ random variables $X_1,X_2,\dots$
have a bounded support, i.e., the distribution measure is in
$\cP^\infty(M)$.
\end{remark}

\begin{example}\label{E-5.5}\rm
Let $X_1,X_2,\dots$ be a sequence i.i.d.\ random variables with
values in a finite set $\{A_1,\dots,A_K\}$ in
$\mathbb{P}=\mathbb{P}(\cH)$, whose distribution is
$\mu_0=\sum_{j=1}^Kw_j\delta_{A_j}$, where $w_j>0$ and $\sum_{j=1}^Kw_j=1$.
Let $G$ be the Karcher barycenter on $\bbP$, and consider the
$G$-value of the empirical measure
$$
G\Biggl({1\over n}\sum_{i=1}^n\delta_{X_i(\omega)}\Biggr)
=G(X_1(\omega),\dots,X_n(\omega)).
$$
By Theorem \ref{T-5.2} the distribution of $\bbP$-valued random variable
$G(X_1(\omega),\dots,X_n(\omega))$ satisfies the LDP with the good
rate function
$$
I(A):=\inf\Biggl\{\sum_{j=1}^Kp_j\log{p_j\over w_j}:
A={G\Biggl(\sum_{j=1}^Kp_j\delta_{A_j}\Biggr)}\Biggr\}
\quad\mbox{for $A\in\bbP$}.
$$
Let $\Delta_K$ be the set of all $K$-dimensional probability
vectors, and let
$$
\Gamma_G(A_1,\dots,A_K):=\Biggl\{G\Biggl(\sum_{j=1}^Kp_j\delta_{A_j}\Biggr):
(p_1,\dots,p_K)\in\Delta_K\Biggr\}.
$$
Assume that $A_1,\dots,A_K$ are ``in general position" with
respect $G$ in the sense that $(p_1,\dots,p_K)\in\Delta_K\mapsto
G\bigl(\sum_{j=1}^Kp_j\delta_{A_j}\bigr)\in\bbP$ is
one-to-one. In this case, the above rate function is written as
$$
I(A)=\begin{cases}\sum_{j=1}^Kp_j\log{p_j\over w_j}
& \mbox{if $A\in\Gamma_G(A_1,\dots,A_K)$ and $A=G\bigl(\sum_{j=1}^Kp_j\delta_{A_j}\bigr)$}, \\
\infty & \mbox{if $A\not\in\Gamma_G(A_1,\dots,A_K)$}.
\end{cases}
$$
\end{example}

\section{Acknowledgements}

The work of F.~Hiai was supported in part by Grant-in-Aid for
Scientific Research (C)17K05266. The work of Y. Lim was supported by
the National Research Foundation of Korea (NRF) grant funded by the
Korea government (MEST) No.NRF-2015R1A3A2031159.

\end{document}